\documentclass{emsprocart}

\newtheorem{thm}{Theorem}[section]
\newtheorem{cor}[thm]{Corollary}
\newtheorem{lemma}[thm]{Lemma}
\newtheorem{prop}[thm]{Proposition}
\newtheorem{conj}[thm]{Conjecture}

\theoremstyle{definition}
\newtheorem{df}[thm]{Definition}
\newtheorem{rem}[thm]{Remark}
\newtheorem{ex}[thm]{Example}

\usepackage{amsmath,amssymb,amscd,graphics,enumerate,latexsym,verbatim,stmaryrd,graphicx,color,makeidx}
\usepackage{tocloft}
\usepackage[all]{xy}

\renewcommand{\theenumi}{\roman{enumi}}                      

\DeclareMathOperator{\GL}{GL}
\DeclareMathOperator{\SL}{SL}

\DeclareMathOperator{\SO}{SO}

\DeclareMathOperator{\Sp}{Sp}

\DeclareMathOperator{\Gr}{Gr}

\DeclareMathOperator{\spec}{spec}
\DeclareMathOperator{\Spec}{Spec}
\DeclareMathOperator{\CSpec}{CSpec}
\DeclareMathOperator{\Sch}{Sch}
\DeclareMathOperator{\Comm}{Comm}
\DeclareMathOperator{\Aff}{Aff}

\DeclareMathOperator{\Frob}{Frob}
\DeclareMathOperator{\res}{res}
\DeclareMathOperator{\Hom}{Hom}

\DeclareMathOperator{\Cl}{Cl}
\DeclareMathOperator{\Bun}{Bun}
\DeclareMathOperator{\Proj}{Proj}

\DeclareMathOperator{\Norm}{Norm}
\DeclareMathOperator{\Cent}{Cent}

\DeclareMathOperator{\colim}{colim}

\def\Mod{\mathcal{M}od \,}
\def\0{{\bf 0}}
\def\A{{\mathbb A}}
\def\B{{\mathbb B}}
\def\C{{\mathbb C}}
\def\F{{\mathbb F}}
\def\G{{\mathbb G}}

\def\N{{\mathbb N}}
\def\P{{\mathbb P}}
\def\Q{{\mathbb Q}}
\def\R{{\mathbb R}}
\def\S{{\mathbb S}}
\def\T{{\mathbb T}}
\def\Z{{\mathbb Z}}
\def\cA{{\mathcal A}}
\def\cB{{\mathcal B}}
\def\cC{{\mathcal C}}
\def\cD{{\mathcal D}}

\def\cF{{\mathcal F}}
\def\cG{{\mathcal G}}

\def\cI{{\mathcal I}}
\def\cK{{\mathcal K}}
\def\cL{{\mathcal L}}
\def\cM{{\mathcal M}}
\def\cN{{\mathcal N}}
\def\cO{{\mathcal O}}

\def\cP{{\mathcal P}}

\def\cR{{\mathcal R}}
\def\cS{{\mathcal S}}
\def\cT{{\mathcal T}}
\def\cU{{\mathcal U}}

\def\cW{{\mathcal W}}
\def\cX{{\mathcal X}}
\def\cY{{\mathcal Y}}

\def\fe{{\mathfrak e}}
\def\fg{{\mathfrak g}}

\def\fm{{\mathfrak m}}
\def\fp{{\mathfrak p}}

\def\Fun{{\F_1}}
\def\Funsq{{\F_{1^2}\!}}
\def\Funn{{\F_{1^n}\!}}

\def\int{\textup{int}}

\def\id{\textup{id}}
\def\1{\textbf{1}}

\def\barx{{\overline{x}}}
\def\bary{{\overline{y}}}
\def\overz{{\overline{\Spec\Z}}}
\def\overok{{\overline{\Spec O_K}}}

\def\blanc{-}
\def\bp{{\mathcal{B}lpr}}
\def\bpspaces{{\mathcal{L}\!oc\bp \mathcal{S}\!p}}
\def\canc{{\textup{canc}}}

\def\rk{{\textup{rk}}}

\def\min{{\textup{min\,}}}
\def\SRings{{\mathcal{S\!R}\textit{ings}}}
\def\Rings{{\mathcal{R}\textit{ings}}}
\def\Sets{\textup{Sets}}

\def\={\equiv}
\def\n={\equiv\hspace{-10,5pt}/\hspace{3,5pt}}

\def\un{\underline{\bf n}}

\def\px{\star} 

\def\st{{\textup{st}}}
\def\et{{\textup{\'et}}}
\def\bl{{\textup{bl}}}
\def\top{{\textup{top}}}
\def\alg{{\textup{alg}}}
\def\log{{\textup{log}}}
\def\hom{{\textup{hom}}}
\def\Sotimes{\otimes^{\!+}}

\def\toptimes{\times^\top}
\def\SA{{\vphantom{\A}}^{+\!}{\A}}   
\def\SP{{\vphantom{\P}}^{+}{\P}}   
\def\SG{{\vphantom{\G}}^{+}{\G}}   
\def\sT{{\scriptscriptstyle\cT\hspace{-3pt}}}

\def\ud{{\underline{d}}}
\def\ue{{\underline{e}}}

\def\un{{\underline{n}}}
\def\udim{{\underline{\dim}\, }}

\DeclareMathOperator{\BSch}{Sch_\Fun}

\DeclareMathOperator{\TSch}{Sch_\cT}

\newdir{ >}{{}*!/-5pt/@^{(}}\newcommand{\arincl}[1]{\ar@{ >->}@<-0,0ex>#1} 

\newcommand{\norm}[1]{\left| #1 \right|}

\newcommand{\gen}[1]{\langle #1 \rangle}

\newcommand{\bpquot}[2]{#1\!\sslash\!#2}
\newcommand{\bpgenquot}[2]{#1\!\sslash\!\gen{#2}}

\newcommand{\tinymat}[4]{\bigl( \begin{smallmatrix} #1 & #2 \\ #3 & #4 \end{smallmatrix} \bigr)}

\contact[lorschei@impa.br]{Oliver Lorscheid, Instituto Nacional de Matem\'atica Pura e Aplicada, Estrada Dona Castorina 110, Rio de Janeiro, Brazil}

\makeindex

\title{A blueprinted view on $\Fun$-geometry}

\author[Oliver Lorscheid]{Oliver Lorscheid}

\begin{document}

\begin{abstract}
 This overview paper has two parts. In the first part, we review the development of $\Fun$-geometry from the first mentioning by Jacques Tits in 1956 until the present day. We explain the main ideas around $\Fun$, embedded into the historical context, and give an impression of the multiple connections of $\Fun$-geometry to other areas of mathematics.

 In the second part, we review (and preview) the geometry of blueprints. Beyond the basic definitions of blueprints, blue schemes and projective geometry, this includes a theory of Chevalley groups over $\Fun$ together with their action on buildings over $\Fun$; computations of the Euler characteristic in terms of $\Fun$-rational points, which involve quiver Grassmannians; $K$-theory of blue schemes that reproduces the formula $K_i(\Fun)=\pi^\st_i(\S^0)$; models of the compactifications of $\Spec \Z$ and other arithmetic curves; and explanations about the connections to other approaches towards $\Fun$ like monoidal schemes after Deitmar, $\B_1$-algebras after Lescot, $\Lambda$-schemes after Borger, relative schemes after To\"en and Vaqui\'e, log schemes after Kato and congruence schemes after Berkovich and Deitmar.
\end{abstract}

\begin{classification}

\end{classification}

\begin{keywords}
 ABC-conjecture; Arakelov theory; blueprint; building; Chevalley group; completion of $\Spec\Z$; congruence scheme; cyclotomic field extension; Euler characteristic; $\Fun$-geometry; $K$-theory; $\Lambda$-ring; log scheme; monoid; projective geometry; quiver Grassmannian; Riemann hypothesis; simplicial complex; stable homotopy of spheres; toric variety; total positivity; Weyl group; zeta function.
\end{keywords}

\maketitle

\setlength\cftbeforesecskip{0pt}
\setlength\cftbeforepartskip{20pt}
{\small \tableofcontents}

\part*{Prologue}
\addcontentsline{toc}{part}{Prologue}

\noindent
Throughout the last two decades, in which the field of $\Fun$-geometry has been formed and established, the most prominent phrases to comprise the purpose of a \emph{field with one element} were probably the following two. The Weyl group\index{Weyl group} $W$ of a Chevalley group\index{Chevalley group} $G$ should be the group of $\Fun$-rational points\index{F1-rational points@$\Fun$-rational points}, i.e.\ 
\[
 G(\Fun) \quad = \quad W,
\]
and the completion of $\Spec\Z$\index{Completion of $\Spec O_K$} at the infinite place should be a curve over $\Fun$; in particular, $\Z$ should be an algebra over $\Fun$.

These ideas inspired and guided the search for a suitable geometry over $\Fun$, though these two concepts are, strictly speaking, incompatible. This is quickly explained (see the introduction of \cite{L12} for more details). The $\Fun$-algebra $\Z$ defines a base extension functor $\blanc\otimes_\Fun\Z$ and a map
\[
 \sigma: \ W \ = \ G(\Fun)\ = \ \Hom(\Spec\Fun,G_\Fun) \quad \stackrel{\blanc\otimes_\Fun\Z}{\longrightarrow} \quad \Hom(\Spec\Z,G) \ = \ G(\Z)
\]
where $G_\Fun$ is an $\Fun$-model of $G$. Since $\Fun$-geometry rigidifies usual geometry, we expect that the $\Fun$-model $G_\Fun$ fixes the choice of a maximal split torus $T$ in $G(\Z)$ and an isomorphism $W=N/T$ where $N$ is the normalizer of $T$. It is natural to assume that $\sigma$ is a group homomorphism and a section to the defining sequence
\[
 \xymatrix@C=4pc{1\ar[r]&T\ar[r]&N\ar[r]&W\ar[r]\ar@{-->}@/_1pc/[l]_\sigma&1.}
\]
of the Weyl group $W$. This is, however, not possible in general as the following example shows. Namely, the coset $\tinymat 01{-1}0 T$ of the normalizer $N=T\cup \tinymat 01{-1}0 T$ of the diagonal torus $T$ of $G(\Z)=\SL_2(\Z)$ does not contain an element of order $2$ while the Weyl group $W=\{\pm 1\}$ does.

Before we explain in Part \ref{partII} how to resolve this problem and why blueprints are the natural framework for Chevalley groups over $\Fun$, let us have a look at the development of these ideas and, more generally, of $\Fun$-geometry from its beginning until the present day. In particular, we would like to emphasize that nowadays, $\Fun$-geometry has a plenitude of connections to other fields of mathematics.


\part{A rough guide to $\Fun$}
\label{partI}

\section*{First milestones}
\addcontentsline{toc}{section}{First milestones}

The first speculation about the existence of a ``field of characteristic one'' can be found in the 1956 paper \cite{Tits56} by Jacques Tits. He observed that certain incidence geometries that come from point--line configurations in a projective space over a finite field $\F_q$ have a combinatorial counterpart when $q$ goes to $1$. This means that there exist \emph{thin} incidence geometries that satisfy similar axioms as the \emph{thick} geometries that come from projective spaces over $\F_q$. This picture is compatible with the action of algebraic groups on the thick incidence geometries and of their combinatorial counterparts, the Weyl groups, on the thin incidence geometries. This inspired Tits to dream about an explanation of the thin geometries as a geometry over $\Fun$. 

All this finds a nice formulation in terms of the group $G(\F_q)$ of $\F_q$-rational points of a Chevalley group $G$ that acts on a spherical building $\cB_q$\index{Building}. The limit $q\to 1$ is the Weyl group $W$ of $G$ that acts on the apartment\index{Apartment} $\cA=\cB_1$ of $\cB_q$. This has led to the slogan $G(\Fun)=W$. For more details on this, see \cite{Cohn04}, \cite{Thas13} and Section \ref{section: buildings} of this text.

Much later, in the early 1990's, the expected shape of $\Fun$-geometry took a more concise form when it was connected to profound number theoretic problems. Though it is hard to say who contributed to which idea in detail, one finds the following traces in the literature.

Smirnov describes in \cite{Smirnov92} an approach to derive the ABC-conjecture from an approximate Hurwitz formula for a morphism $\overz\to\P^1_\Funsq$, for which he gives an ad hoc definition. Here $\Funsq=\{0,\pm 1\}$ is the \emph{constant field} of $\Spec\Z$. More generally, he considers any \emph{arithmetic curve $X$}\index{Arithmetic curve}, i.e.\ the Arakelov compactification of the spectrum of the algebraic integers in a number field $K$, and conjectures an approximate Hurwitz formula for the morphism $X\to\P^1_{\Funn}$ where $\Funn=\{0\}\cup\mu_n$ is the union of $0$ with all $n$-th roots of unity of $K$.

Around the same time, Kapranov and Smirnov aim in \cite{Kapranov-Smirnov} to calculate cohomological invariants of arithmetic curves in terms of cohomology over $\Funn$. The unfinished text contains many concepts that were of influence to the further development of $\Fun$-geometry: linear and homological algebra over $\Funn$, distinguished morphisms as cofibrations, fibrations and equivalences (which might be seen as a first hint of the connections of $\Fun$-geometry to homotopy theory), Arakelov theory \emph{modulo $n$}\index{Arakelov theory} and connections to class field theory and reciprocity laws, which can be seen in analogy to knots and links in $3$-space. See Manin's paper \cite{Manin06} for further explanations on the connection between arithmetic curves and $3$-manifolds.
Kurokawa (\cite{Kurokawa92}) publishes in 1992 a construction of absolute tensor products of zeta functions of objects of different characteristics. The underlying assumption is that these objects are all defined over an absolute point, i.e.\ $\Fun$. In his 1995 lecture notes \cite{Manin95}, Manin ties up this viewpoint with Deninger's work \cite{Deninger91}, \cite{Deninger92} and \cite{Deninger94} on a conjectural formalism of motives that aims at transferring Weil's proof (\cite{Weil48}) of the Riemann hypothesis from positive characteristic to the classical case. This formalism includes a conjectural decomposition 
\begin{multline*} 
2^{-1/2}\pi^{-s/2}\Gamma(\frac s2)\zeta(s) \ = \\ \frac{\det_\infty\Bigl(\frac 1{2\pi}(s-\Theta)\Bigl| H^1(\overline{\Spec\Z},\cO_\cT)\Bigr.\Bigr)}{\det_\infty\Bigl(\frac 1{2\pi}(s-\Theta)\Bigl| H^0(\overline{\Spec\Z},\cO_\cT)\Bigr.\Bigr)\det_\infty\Bigl(\frac 1{2\pi}(s-\Theta)\Bigl| H^2(\overline{\Spec\Z},\cO_\cT)\Bigr.\Bigr)} 
\end{multline*}
of the completed Riemann zeta function where the two \emph{regularized determinants $\det_\infty(\dotsc)$} of the denominator equal $s/2\pi$ resp.\ $(s-1)/2\pi$. Manin observes that these functions fit Kurokawa's formalism if $s/2\pi$ is interpreted as $\zeta_{\T^0}(s)$ and $(s-1)/2\pi$ is interpreted as $\zeta_\T(s)$ where $\T$ is the postulated \emph{absolute Tate motive}\index{Motive}, or, following Kurokawa's line of thought, the \emph{affine line} $\A^1_\Fun$ over the \emph{absolute point} $\T^0=\Spec\Fun$.


\section*{A dozen ways to define $\Fun$}
\addcontentsline{toc}{section}{A dozen ways to define $\Fun$}

The first suggestion of what a \emph{variety over $\Fun$} should be was given by Soul\'e (\cite{Soule04}) in 2004. This paper was of particular importance since it adds the following ideas to $\Fun$-geometry. It contains the definition of the zeta function $\zeta_X(s)$\index{Zeta function} of an $\Fun$-scheme $X$ whose $\F_q$-rational points $X(\F_q)$ can be counted by a polynomial\index{Counting polynomial} 
\[
 N_X(q) \quad = \quad \sum_{i=0}^n \ a_i q^i
\]
with integral coefficients $a_i$. A priori, $\zeta_X(s)$ is defined as the inverse of the limit $q\to 1$ of the zeta function $\zeta(X_{\F_q},s)$ over $\F_q$ where one has to multiply by $(q-1)^{N(1)}$ to resolve the pole at $q=1$. Note that its order $N(1)$ might be interpreted as the number of $\Fun$-rational points\index{F1-rational points@$\Fun$-rational points} of $X$. The comparison theorem for singular and $l$-adic cohomology and the Weil conjectures imply that $N(1)$ equals the Euler characteristic of $X_\C$ if $X_\Z$ is smooth and projective. See Section \ref{section: euler cahracteristics and f1-rational points} of this text for more on the connection between the Euler characteristic and $\Fun$-rational points.

It turns out that the limit $q\to 1$ of $(q-1)^{-N(1)}\zeta(X_{\F_q},s)^{-1}$ is the elementary function
\[
 \zeta_X(s) \quad = \quad \prod_{i=0}^n \ (s-i)^{a_i}.
\]
In particular, $\zeta_{\Spec\Fun}(s)=s$ and $\zeta_{\A^1_\Fun}(s)=s-1$, which equals, up to a factor $1/2\pi$, the motivic zeta functions of Deninger, as postulated by Manin. The advantage of this definition is that it depends only on the number of $\F_q$-rational points of $X$, but not on the particular theory of $\Fun$-varieties, so that this notion of $\zeta$-functions applies to every notion of an $\Fun$-scheme.

Therefore it was possible to investigate zeta functions of $\Fun$-schemes in a down to earth manner and largely independent from the development of scheme theory over $\Fun$. In particular, Kurokawa explores, partly in collaborations with others, properties of absolute zeta functions and developed other notions for absolute arithmetics like derivations, modular forms and Hochschild cohomology, cf.\ his papers \cite{Kurokawa04} and \cite{Kurokawa05} and his collaborations \cite{Deitmar-Koyama-Kurokawa08} (with Deitmar and Koyama), \cite{Kim-Koyama-Kurokawa09} (with Kim and Koyama), \cite{Koyama-Kurokawa04}, \cite{Koyama-Kurokawa10}, \cite{Koyama-Kurokawa11a} and \cite{Koyama-Kurokawa11b} (with Koyama), \cite{Kurokawa-Ochiai-Wakayama03} (with Ochiai and Wakayama) and \cite{Kurokawa-Wakayama04} (with Wakayama). Other works on absolute zeta functions are Connes and Consani's paper \cite{Connes-Consani11c}, Deitmar's paper \cite{Deitmar07}, Minami's paper \cite{Minami09} and the author's paper \cite{L10}. Further 
work 
on Hochschild cohomology appears in Betley's paper \cite{Betley10}.

Soul\'e's paper \cite{Soule04} also contains the hypothetical formula\index{Stable homotopy of spheres}
\[
 K_*(\Spec\Fun) \quad = \quad \pi_*({B\GL(\infty,\Fun)}^+) \quad  = \quad \pi_*({BS_\infty}^+) \quad \simeq  \quad \pi_*^\st({\S^0}) 
\]
where the first equality is the definition of $K$-theory\index{K-theory@$K$-theory} via Quillen's +-construc\-tion, symbolically applied to the elusive field $\Fun$. The equality in the middle is derived from Tits' idea $\GL(n)=S_n$, and therefore
\[ 
 \GL(\infty,\Fun) \quad = \quad \bigcup_{n\geq 1} \GL(n,\Fun) \quad = \quad \bigcup_{n\geq 1} S_n \quad = \quad S_\infty. 
\]
The last isomorphism is the Barratt-Priddy-Quillen theorem (\cite{Barratt71}, \cite{Priddy71}).

One might speculate about the psychological impact of this first attempt to define $\Fun$-geometry: soon after, a dozen further approaches towards $\Fun$ (and more) appeared. Without keeping the chronological order, we mention them in the following. For a brief introduction into the $\Fun$-geometries that appeared until 2009, see the overview paper \cite{LL11b} of L\'opez Pe\~na and the author. 

Soul\'e's definition underwent a number of variations by himself in \cite{Soule11} and by Connes and Consani in \cite{Connes-Consani11a} and later in \cite{Connes-Consani10a}. The latter notion of an $\Fun$-scheme turns out to be in essence the same as a torified scheme as introduced by L\'opez Pe\~na and the author in \cite{LL11a} in order to produce examples for the $\Fun$-varieties considered in \cite{Soule04} and \cite{Connes-Consani11a}.

Deitmar reinterprets in \cite{Deitmar05} Kato's generalization (\cite{Kato94}) of a toric variety\index{Toric variety} as an $\Fun$-scheme. Deitmar introduces the notion of a prime ideal for a monoid, which yields a topological space of a combinatorial flavour together with a structure sheaf in monoids. Deitmar pursues the theory in the papers \cite{Deitmar07} and \cite{Deitmar08}. This category of $\Fun$-schemes is minimalistic in the sense that it can be embedded into any other $\Fun$-theory. In this sense, toric geometry and its generalization by Kato and Deitmar constitutes the very core of $\Fun$-geometry.

Durov's view on Arakelov theory\index{Arakelov theory} in terms of monads in his comprehensive thesis \cite{Durov07} produces a notion of $\Fun$-schemes. See Fresan's m\'emoire \cite{Fresan09} for a summary of this theory. Of particular interest is that Durov's theory contains a model of $\overz$\index{Completion of $\Spec O_K$}. Similarly, Haran's definition of a generalized scheme in \cite{Haran07} yields an $\Fun$-geometry with a model for $\overz$. Haran modifies this approach in \cite{Haran09} and yields a model for $\overz$ with an interesting self-product $\overz\times\overz$ over $\Fun$. Further models of $\overz$ are given in Takagi's paper \cite{Takagi12b}, which is based on his previous work \cite{Takagi10}, \cite{Takagi11a}, \cite{Takagi11b} and \cite{Takagi12a}, and in the author's paper \cite{blueprints-mpi}. See Section \ref{section: the arithmetic line} for a description of the latter model of $\overz$ as a locally blueprinted space.

To\"en and Vaqui\'e generalize in \cite{Toen-Vaquie09} the functorial viewpoint on scheme theory to any closed complete and cocomplete symmetric monoidal category $\cC$ that replaces the category of $R$-modules in the case of schemes over a ring $R$. They call the resulting objects \emph{schemes relative to $\cC$}\index{Relative scheme}. An $\Fun$-scheme is a scheme relative to the category of sets, together with the Cartesian product. After partial results by Marty in \cite{Marty07}, Vezzani shows in \cite{Vezzani12} that this notion of an $\Fun$-scheme is equivalent to Deitmar's. Marty develops smooth morphism for relative schemes in \cite{Marty08}. In \cite{L12b}, the author shows that blue schemes can be realized as schemes relative to certain module categories; for a summary, see Section \ref{subsection: relative schemes after toen and vaquie}.

Borger's work \cite{Borger11a} and \cite{Borger11b} on $\Lambda$-algebraic geometry\index{Lambda-scheme@$\Lambda$-scheme} yields a notion of an $\Fun$-scheme in \cite{Borger09}. See Section \ref{subsection: lambda-schemes} for more details and the connection to blueprints. 

Lescot develops in \cite{Lescot09}, \cite{Lescot11}, \cite{Lescot12} and \cite{Lescot12b} a geometry associated with idempotent semirings, which he gives the interpretation of an $\Fun$-geometry. See Section \ref{subsubsection: Idempotent semirings and sesquiads} for the connection to blueprints. 

Connes and Consani extend Lescot's viewpoint to the context of Krasner's hyperrings (\cite{Krasner57} and \cite{Krasner83}), which they promote as a geometry over $\Fun$ in \cite{Connes-Consani10b} and \cite{Connes-Consani11b}.

Berkovich introduces a theory of congruence schemes\index{Congruence scheme} for monoids (see \cite{Berkovich11}), which can be seen as an enrichment of Deitmar's $\Fun$-geometry. Deitmar modifies this approach and applies it to sesquiads, which is a common generalization of rings and monoids. See Section \ref{subsubsection: Idempotent semirings and sesquiads} for the connection of sesquiads to blueprints and Section \ref{subsection: congruence schemes} for further remarks on congruence schemes. 

The author develops in the papers \cite{blueprints1}, \cite{blueprints2}, \cite{LL12} (jointly with L\'opez Pe\~na), \cite{blueprints-mpi} and \cite{L12b} the theory of blueprints\index{Blueprint} and blue schemes\index{Blue scheme}, which will be reviewed in Part \ref{partII} of this paper.


\section*{Results and new directions}
\addcontentsline{toc}{section}{Results and new directions}

In the last three years, the number of publications on $\Fun$ has more than doubled. Many new ideas were formed and many connections to other fields of mathematics were found and established. We will give an overview of recent results around $\Fun$. This overview is probably not exhaustive, but hopefully serves as an impression of the versatility of $\Fun$ today. 

We illustrate in Figure \ref{figure: applications} some connections to other fields of mathematics, which can be sorted roughly into the four branches arithmetic, homotopy theory, geometry and combinatorics. 

\begin{figure}[t]
 \includegraphics{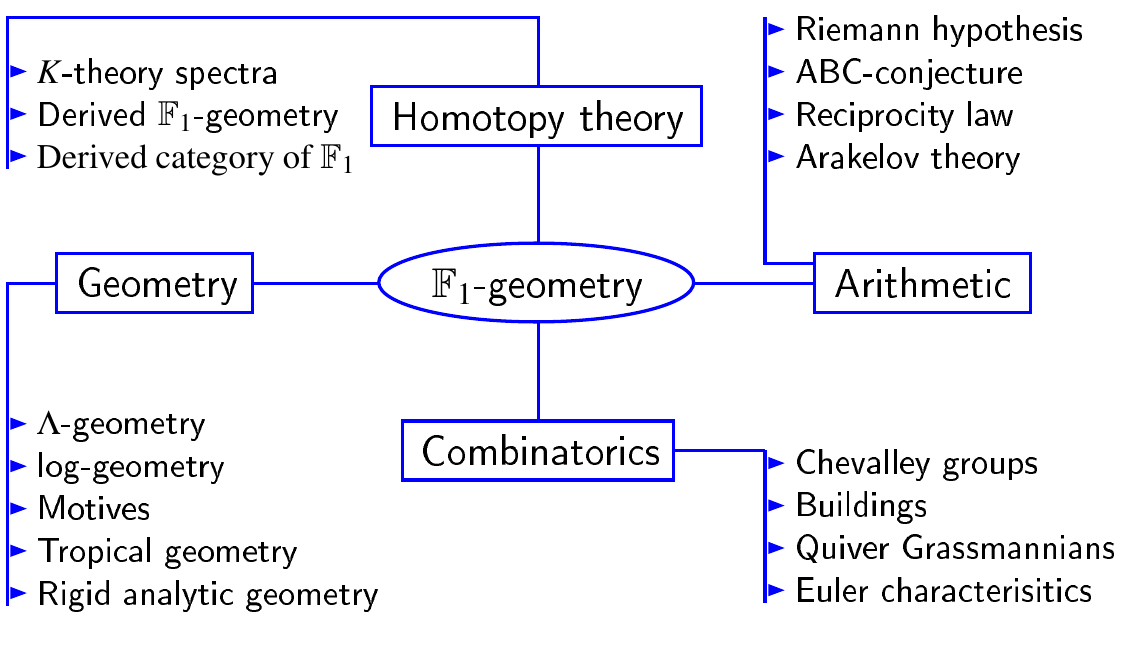}
 \caption{The connections of $\Fun$-geometry}\label{figure: applications}
\end{figure}

\subsection*{Arithmetic}

As explained in the previous sections, there are hopes to approach the Riemann hypothesis and the ABC-conjecture with $\Fun$-geometry. While the former remains unsolved, Mochizuki (\cite{Mochizuki12a}, \cite{Mochizuki12b}, \cite{Mochizuki12c}, \cite{Mochizuki12d}) has recently announced a proof of the ABC-conjecture. Though it does not seem to be connected to Smirnov's idea of a conjectural Hurwitz inequality for $\overz\to\P^1_{\Funsq}$, it borrows ideas from $\Fun$-geometry as occasional remarks (e.g.\ \cite[Remark 3.12.4 (iv)]{Mochizuki12c}) make clear.

The connection of the different models of $\overz$ to Arakelov theory\index{Arakelov theory} is apparent. However, the algebraically defined vector bundles on $\overz$ come naturally equipped with the $1$-norm (or, dually, with the $\infty$-norm) at the archimedean places while Arakelov geometry considers vector bundles together with the $2$-norm. This digression has not yet been explained in a satisfying way. We remark that the definition of $\overz$ as a locally blueprinted space in \cite{blueprints-mpi} yields a reinterpretation of the Euler factors of the Riemann zeta function, including the Gamma factor, in terms of integrals over the ideal space of the local blueprints.

\subsection*{Homotopy theory}

There are several connections of $\Fun$-geometry to homotopy theory. The expected isomorphism of the $K$-theory\index{K-theory@$K$-theory} of $\Fun$ with the stable homotopy of the sphere spectrum\index{Stable homotopy of spheres} has been established. Deitmar defines in \cite{Deitmar07} $K$-groups for commutative monoids and derives $K_i(\Fun)=\pi^\st_i(\S^0)$. This has been generalized by Chu and Morava (\cite{Chu-Morava10}) to non-commutative monoids, with a $G$-theory variant by Mahanta (\cite{Mahanta11}), and by Chu, Santhanam and the author (\cite{CLS12}) to monoidal schemes. We review the latter theory in the more general context of blue schemes in Section \ref{section: k-theory} of this paper. Calculations of vector bundles on $\P^n_\Fun$, Grothendieck groups $K_0(\P^n_\Fun)$ and Hall algebras can be found in \cite{Bothmer-Hinsch-Stuhler11} (von Bothmer, Hinsch and Stuhler), \cite{Szczesny12b} and \cite{Szczesny12c} (Szczesny) and the mentioned paper \cite{CLS12}.

Salch gives in \cite{Salch10} a definition of the bounded below derived category $D^+(\Fun)$ of $\Fun$ and shows that this category is equivalent to pointed simplicial sets. This might be interpreted as the Dold-Kan correspondence over $\Fun$. Consequently, one can interpret the stable homotopy category as the derived category over $\Fun$. Anevski reconstructs in \cite{Anevski11} the spectrum of $\Fun$ from the stable homotopy category. 

Banerjee applies in \cite{Banerjee12} Lurie's machinery of derived algebraic geometry (\cite{LurieI}, \cite{LurieII}, \cite{LurieIII}, et cetera) to define derived schemes over $\Fun$. 

Corti\~nas, Haesemeyer, Mark, Walker and Weibel extend in \cite{CHMWW11} and \cite{CHMWW12} the theory of monoidal schemes and apply it to calculate cdh-descent and positive characteristic $K$-theory for toric varieties. This can be seen as a hint of the role of motivic cohomology for the future of $\Fun$.

Finally we would like to mention that homotopy theory might be the only way to define a good substitute for sheaf cohomology. Note that Deitmar shows in \cite{Deitmar11b} that the methods of homological algebra extend to $\cO_X$-modules $\cM$ of an $\Fun$-scheme $X$, which yields a definition of cohomology \emph{sets} $H^i(X,\cM)$. A calculation of sheaf cohomology for $X=\P^1_\Fun$ shows that $H^1(\P^1_\Fun,\cO_X)$ is infinite-dimensional (see \cite{p1-calculation}). There is some hope that \v{C}ech cohomology produces smaller cohomology groups, but it is also possible that the only path towards (finite-dimensional) cohomology passes through homotopy theory.

\subsection*{Geometry}

Nowadays, there is a plenitude of highly developed geometric categories that were invented for one or the other reason. It might be the aim of a larger program to unify these geometries to a theory with common roots. Many signs hint towards homotopy theory, but a reinterpretation in terms of topological spectra comes with the loss of all the algebraic structure that makes ``algebraic geometries'' very manageable and computable. In some special cases, much can be seen already from the $\Fun$-perspective.

$\Lambda$-algebraic geometry\index{Lambda-scheme@$\Lambda$-scheme} after Borger (\cite{Borger09}) provides its own definition of an $\Fun$-scheme. Since a $\Lambda$-structure of a scheme is thought of as a descend datum to $\Fun$, one might change the viewpoint and impose that the base extension of an $\Fun$-scheme should come with a $\Lambda$-structure. We exploit this viewpoint in Section \ref{subsection: lambda-schemes} of this text. Similarly, logarithmic algebraic geometry\index{Log scheme} is concerned with schemes with extra structure, namely, a sheaf of monoids. It turns out that a log scheme is indeed closely connected to a blue scheme, see Section \ref{subsection: log schemes} of this text.

Following the line of thoughts of Manin in \cite{Manin95}, there should be a link between $\Fun$-schemes and motives\index{Motive} that explains the similar structure of their zeta functions\index{Zeta function}. There are further connections of motives and $\Fun$-schemes to non-commutative geometry, see Connes, Consani and Marcolli's papers \cite{Connes-Consani-Marcolli07} and \cite{Connes-Consani-Marcolli09}, Bejleri and Marcolli's paper \cite{Bejleri-Marcolli12} and Marcolli's appendix of Sujatha and Plazas' lecture notes \cite{Sujatha-Plazas11}.

Another theory of interest is analytic geometry over $\Fun$. Manin describes in \cite{Manin10} analytic geometry over $\Fun$ in terms of Habiro rings and connects this to certain power series in $q$ whose domains of convergence have a particular behaviour at roots of unity, i.e.\ over $\Funn$. 

It is further expected that analytic geometry over $\Fun$ connects to analytic geometry after Berkovich (\cite{Berkovich93} and \cite{Berkovich98}) and Huber (\cite{Huber96}) and to tropical geometry (see, for instance, \cite{Mikhalkin06}). In particular, Berkovich defines in \cite{Berkovich11} congruence schemes for monoids in order to explain skeletons of Berkovich spaces. Since the monoids that appear in Berkovich's theory are the underlying monoids of idempotent semirings and since tropical schemes are also defined in terms of idempotent semirings (see Mikhalkin's unpublished notes \cite{Mikhalkin10}), it is desirable to extend the notion of congruence schemes to semirings. There is hope for a theory of congruence schemes for blueprints, which contains the special cases of monoids with zero (Berkovich), sesquiads (Deitmar) and semirings. See Section \ref{subsection: congruence schemes} for more explanations on this. 

It seems that blueprints match and generalize Paugam's approach to global analytic geometry in \cite{Paugam09}, which yields a conceptual explanation of Berkovich spaces and Huber's adic spaces. Therefore one might hope for a functorial connection between analytic spaces over $\Fun$ on one side and Berkovich spaces, adic spaces and tropical schemes on the other side.

\subsection*{Combinatorics}

The combinatorial flavour of $\Fun$-geometry gives access to geometric explanations of combinatorial phenomena, and this might be used to develop computational methods for usual algebraic geometry.

Tits' idea of Chevalley groups over $\Fun$ and the formula $G(\Fun)=W$ found numerous remarks in the literature on $\Fun$, and this formula seemed to make sense from many viewpoints (see \cite{Borger09}, \cite{Cohn04}, \cite{Connes-Consani10a}, \cite{Connes-Consani11a}, \cite{Deitmar05}, \cite{Deitmar11a}, \cite{Kapranov-Smirnov}, \cite{Kurokawa05}, \cite{LL11a}, \cite{Manin95}, \cite{Soule04} and \cite{Toen-Vaquie09}). For the reason explained in the beginning of this text, it was not possible to combine the formula $G(\Fun)=W$ with a theory of group schemes over $\Fun$. In \cite{L12}, the author describes a theory of group schemes that circumvents the problem by considering two classes of morphisms. The longing for a more conceptual understanding of $\Fun$-models of Chevalley groups gave birth to the notion of a blueprint and the theory of blue schemes (\cite{blueprints1} and \cite{blueprints2}). See Part \ref{partII} and, in particular, Section \ref{section: chevalley groups over fun} for a description of 
this theory.

There are also some remarks on incidence geometries that come from $\Fun$-schemes, see \cite{Cohn04}, \cite{Pirashvili12} and \cite{Thas13}. It turns out that apartments\index{Apartment} of buildings\index{Building} (and other incidence geometries) are naturally attached to representations of Chevalley groups over $\Fun$ on projective space. A description of this is given in Section \ref{section: buildings} of this text.

The hope to calculate cohomological invariants of schemes depends largely on the existence of a (well-behaved) $\Fun$-model. Reineke proves in \cite{Reineke12} that every projective variety can be realized as a quiver Grassmannian. Szczesny defines in \cite{Szczesny12a} quiver representations over $\Fun$, which leads to the notion of $\Fun$-rational points of a quiver Grassmannian $\Gr_\ue(M)$. While this viewpoint might be too restrictive, the data attached to a quiver Grassmannian can be used to define an $\Fun$-model $\Gr_\ue(M)_\Fun$ in the language of blue schemes. In certain cases, the Euler characteristics of $\Gr_\ue(M)$ can be computed by counting certain points of $\Gr_\ue(M)_\Fun$. See Section \ref{section: euler cahracteristics and f1-rational points} for more explanations.

Finally, we would like to mention the following conjecture of Reineke in \cite{Reineke08}. Reineke shows that moduli spaces $X$ of quiver representations have a counting polynomial $N_X(q)$. Conjecture 8.5 in \cite{Reineke08} is that the coefficients $b_i$ of 
\[
 N_X(q) \quad = \quad \sum_{i=0}^n \ b_i\,(q-1)^i
\]
are non-negative. This might be proven by finding a suitable $\Fun$-model for the moduli space $X$.


\part{The geometry of blueprints}
\label{partII}

\section*{Introduction}
\label{section: introduction}
\addcontentsline{toc}{section}{Introduction}

In this second part, we introduce blueprints and blue schemes, and we describe some of their applications. To make this text accessible to the non-expert, we include the basic definitions and aim for a non-technical and partly simplified presentation.

A blueprint can be seen as a multiplicatively closed subset $A$ (with $0$ and $1$) of a semiring $R$. The addition of the semiring allows us to express relations $\sum a_i=\sum b_j$ between sums of elements $a_i$ and $b_j$ of $A$. In the extreme case $A=R$,  the blueprint is a semiring, and in the other extreme case $A\subset R$ where $R=\Z[A]/(0)$ is the semigroup ring of $A$ (with the zeros of $A$ and $R$ identified), the blueprint behaves like a multiplicative monoid. 

Its geometric counterpart, the notion of a blue scheme, includes monoidal schemes and usual schemes as subclasses. Beyond these well-known objects, it contains also a class of semiring schemes as well as $\Fun$-models for all schemes of finite type. This makes the category of blue schemes flexible enough to treat several problems of $\Fun$-geometry. 

We review the basic definitions in the Section \ref{section: definitions}. The theory behaves largely in complete analogy with usual scheme theory. The only new aspect is that the presheaf of localizations of a blueprint is not a sheaf on the prime spectrum of the blueprint. The required sheafification leads to an equivalence between affine blue schemes with the proper subclass of so-called global blueprints. Once aware of this phenomenon, one can proceed with developing properties of scheme theory, such as a locally algebraic characterization of morphisms or fibre products. We conclude Section \ref{section: definitions} with a description of the $\Proj$-functor for graded blueprints and projective varieties. 

In the following sections, we review and preview a number of application of blueprints to $\Fun$-geometry. We list references for the results in this exposition below.

\subsection*{Combinatorics}
 While monoidal schemes and, more generally, blue schemes reflect naturally the combinatorical structures that are expected in $\Fun$-geometry, morphisms of monoidal resp.\ blue schemes are not suited to capture group laws or to relate to the predicted sets of $\Fun$-rational points. In Section \ref{section: chevalley groups over fun}, we explain the notion of the rank space, which reflects the set of $\Fun$-rational points together with an intrinsic schematic structure. This leads to the class of Tits morphisms, which is flexible enough to descend group laws and group actions to $\Fun$ and whose close relation to the rank space yields the expected set of $\Fun$-rational points in a functorial way. 

 In Section \ref{section: chevalley groups over fun}, we briefly review the notion of a Tits-Weyl model and apply it to split reductive groups and to their parabolic and Levi subgroups. In Section \ref{section: buildings}, we explain how Coxeter complexes and buildings together with the usual group actions appear naturally in terms of representations of Tits-Weyl models. In Section \ref{section: euler cahracteristics and f1-rational points}, we establish for a certain class of integral varieties the connection between the complex Euler characteristic and the number of $\Fun$-rational points.

\subsection*{Arithmetic}
 Arakelov theory gives a clear expectation of how the compactification of $\Spec\Z$ should look like, but the obstacle that the stalk at infinity $\bigl\{a\in\Q\bigl|\norm a_\infty\leq1\bigr\}$ is not a ring requires an extension of scheme theory to make sense of the arithmetic curve $\overz$. In Section \ref{section: the arithmetic line}, we will use blueprints to define candidates for $\overz$ and, more generally, for $\overok$ where $O_K$ denoted a ring of integers in a number field $K$.

\subsection*{Homotopy theory}
 The expected interpretation of the stable homotopy of the sphere spectrum as the $K$-theory of the absolute point can be made rigorous by applying Quillen's $Q$-construction to the category of vector bundles over $\Spec\Fun$. The definition of $K$-theory can be extended to all blue schemes. Some care is required since locally projective sheaves are in general not locally free and since not all monomorphisms (epimorphisms) are kernels (cokernels). In Section \ref{section: k-theory}, we explain all this and review the main results for monoidal schemes.

\subsection*{Geometry}
 Though the different notions of $\Fun$-geometry are, strictly speaking, not equivalent to each other, they might be seen as complementing shades of the same circle of phenomena. In some cases, the connection between different $\Fun$-geometries can be made precise, at least for certain subclasses of $\Fun$-schemes. In Sections \ref{subsubsection: relation to monoids} and \ref{subsubsection: Idempotent semirings and sesquiads}, we show that monoids, sesquiads and $\B_1$-algebras are blueprints. In Section \ref{subsection: lambda-schemes}, we explain that a certain class of $\Lambda$-schemes comes from blue schemes that can be equipped with Frobenius operations. In Section \ref{subsection: relative schemes after toen and vaquie}, we interpret blue schemes as schemes relative to categories of $\cO_X$-modules in the sense of To\"en and Vaqui\'e. In Section \ref{subsection: log schemes}, we construct log schemes from blue schemes and characterize the class of log schemes that we obtain. In particular, all fine 
log schemes that have 
a Zariski local atlas come from blue schemes. In Section \ref{subsection: congruence schemes}, we introduce congruences for blueprints and describe three different approaches towards the definition of a structure sheaf on the space of prime congruences.

\subsection*{References}
The references to this text, as existent, are the following. Sections \ref{subsection: blueprints}--\ref{subsection: globalizations} are covered in the author's paper \cite{blueprints1}. Sections \ref{subsection: projective space} and \ref{subsection: closed subschemes} are from the joint paper \cite{LL12} of L\'opez Pe\~na and the author. Section \ref{section: chevalley groups over fun} summarizes the author's paper \cite{blueprints2}, with the exception of Theorem \ref{thm: chevalley groups over f1}, which is proven in such generality due to an additional idea of Markus Reineke (unpublished). The results of Sections \ref{section: buildings} and \ref{section: euler cahracteristics and f1-rational points} are not yet published. The content of Section \ref{section: the arithmetic line} is contained in the author's text \cite{blueprints-mpi} in the special case of $\overz$. The content of Section \ref{section: k-theory} is contained in the joint paper \cite{CLS12} of Chu, Santhanam and the author in the 
special case of monoidal schemes. The 
content of Sections \ref{subsection: lambda-schemes}, \ref{subsection: log schemes} and \ref{subsection: congruence schemes} is not yet published. Section \ref{subsection: relative schemes after toen and vaquie} summarizes the author's text \cite{L12b}.

\subsection*{Acknowledgements}
I would like to thank James Borger and Christophe Soul\'e for their remarks on parts of this text.


\section{Basic definitions}
\label{section: definitions}

In this section, we introduce the notion of a blueprint and the associated scheme theory, which generalizes both usual schemes and monoidal schemes (in the sense of Deitmar, \cite{Deitmar05}). We will further obtain semiring schemes as a special case, i.e.\ schemes whose sheaf of sections forms semirings, and schemes more akin to $\Fun$-geometry like $\Fun$-models of algebraic groups or Grassmannians. We conclude this first section with a description of projective spaces over a blueprint and projective schemes as closed subschemes of projective space.

\subsection{Blueprints}
\label{subsection: blueprints}
 
\noindent
 By a \emph{monoid with zero}\index{Monoid with zero}, we mean in this text always a multiplicatively written commutative semigroup $A$ with a neutral element $1$ and an absorbing element $0$, which are characterized by the properties $1\cdot a=a$ and $0\cdot a=0$ for all $a\in A$. A \emph{morphism of monoids with zero}\index{Morphism!of monoids with zero} is a multiplicative map $f:A_1\to A_2$ that maps $1$ to $1$ and $0$ to $0$. We denote the category of monoids with zero by $\cM_0$.

\begin{df}
 A \emph{blueprint $B$}\index{Blueprint} is a monoid $A$ with zero together with a \emph{pre-addition}\index{Pre-addition} $\cR$, i.e.\ $\cR$ is an equivalence relation on the semiring $\N[A]=\{\sum a_i|a_i\in A\}$ of finite formal sums of elements of $A$ that satisfies the following axioms (where we write $\sum a_i\=\sum b_j$ whenever $(\sum a_i,\sum b_j)\in\cR$):
\begin{enumerate}
 \item\label{ax1} The relation $\cR$ is additive and multiplicative, i.e.\ if $\sum a_i\=\sum b_j$ and $\sum c_k\=\sum d_l$, then $\sum a_i+\sum c_k\=\sum b_j+\sum d_l$ and $\sum a_ic_k\=\sum b_jd_l$.
 \item\label{ax2} The absorbing element $0$ of $A$ is in relation with the zero of $\N[A]$, i.e.\ $0\=(\text{empty sum})$.
 \item\label{ax3} If $a\= b$, then $a=b$ (as elements in $A$).
\end{enumerate}
 A \emph{morphism $f:B_1\to B_2$ of blueprints}\index{Morphism!of blueprints} is a multiplicative map $f:A_1\to A_2$ between the underlying monoids of $B_1$ and $B_2$ with $f(0)=0$ and $f(1)=1$ such that for every relation $\sum a_i\=\sum b_j$ in the pre-addition $\cR_1$ of $B_1$, the pre-addition $\cR_2$ of $B_2$ contains the relation $\sum f(a_i)\=\sum f(b_j)$. Let $\bp$ be the category of blueprints.
\end{df}

\begin{rem}
 Note that the above definition follows the convention of the papers \cite{blueprints2}, \cite{LL12}, \cite{blueprints-mpi} and \cite{L12b}. In the sense of \cite{blueprints1}, the definition above agrees with a \emph{proper} blueprint \emph{with zero}\index{Blueprint!proper with zero} from \cite{blueprints1}. We call a blueprint in the sense of \cite{blueprints1} a \emph{general blueprint}\index{Blueprint!general} (cf.\ Section \ref{subsection: log schemes}). See Section 1.1 of \cite{blueprints2} for more explanations on this notation.
\end{rem}

In the following, we write $B=\bpquot A\cR$ for a blueprint $B$ with underlying monoid $A$ and pre-addition $\cR$. We adopt the conventions used for rings: we identify $B$ with the underlying monoid $A$ and write $a\in B$ or $S\subset B$ when we mean $a\in A$ or $S\subset A$, respectively. Further, we think of a relation $\sum a_i\=\sum b_j$ as an equality that holds in $B$ (without the elements $\sum a_i$ and $\sum b_j$ being defined, in general). 

Given a set $S$ of relations, there is a smallest equivalence relation $\cR$ on $\N[A]$ that contains $S$ and satisfies \eqref{ax1} and \eqref{ax2}. If $\cR$ satisfies also \eqref{ax3}, then we say that $\cR$ is the pre-addition generated by $S$, and we write $\cR=\gen S$. In particular, every monoid $A$ with zero has a smallest pre-addition $\cR=\gen\emptyset$.

\subsubsection{Relation to rings and semirings}
\label{subsubsection: relation to rings}

The idea behind the definition of a blueprint is that it is a blueprint of a ring (in the literal sense): given a blueprint $B=\bpquot A\cR$, one can construct the ring $B_\Z^+=\Z[A]/I(\cR)$ where $I(\cR)$ is the ideal $\{\sum a_i-\sum b_j\in \Z[A]|\sum a_i\=\sum b_j \text{ in }\cR\}$. We can extend a blueprint morphism $f:B_1\to B_2$ linearly to a ring homomorphism $f_\Z^+:B_{1,\Z}^+\to B_{2,\Z}^+$ and obtain a functor $(\blanc)^+_\Z:\bp\to\Rings$ from blueprints to rings.

Similarly, we can form the quotient $B^+=\N[A]/\cR$, which inherits the structure of a semiring by Axiom \eqref{ax1}. This defines a functor $(\blanc)^+:\bp\to\SRings$ from blueprints to semirings. 

On the other hand, semirings can be seen as blueprints: given a (commutative and unital) semiring $R$, we can define the blueprint $B=\bpquot A\cR$ where $A=R^\bullet$ is the underlying multiplicative monoid of $R$ and $\cR=\{\sum a_i\=\sum b_j|\sum a_i=\sum b_j\text{ in }R\}$. Under this identification, semiring homomorphisms are nothing else than blueprint morphisms, i.e.\ we obtain a full embedding $\iota_\cS:\SRings\to\bp$ from the category of semirings into the category of blueprints. This allows us to identify rings with a certain kind of blueprints, and we can view blueprints as a generalization of rings and semirings. Accordingly, we call blueprints in the essential image of $\iota_\cS$ \emph{semirings} and we call blueprints in the essential image of the restriction $\iota_\cR:\Rings\to\bp$ of $\iota_\cS$ \emph{rings}.

The semiring $B^+$ and the ring $B^+_\Z$ come together with canonical maps $B\to B^+$ and $B\to B_\Z^+$, which are blueprint morphisms.

\begin{rem}
 Note that we use the symbol ``$+$'' whenever we want to stress that the pre-addition is indeed an addition, i.e.\ the blueprint in question is a semiring. Since colimits of blueprints do not preserve additions, we are forced to invent notations like $\Z[T]^+$, $B\Sotimes_\Z C$, $\SA_\Z^n$, $\SP^n_\Z$, $\SG_{m,\Z}^n$ to distinct the constructions in usual algebraic geometry from the analogous constructions for blueprints and blue schemes.
\end{rem}

\subsubsection{Relation to monoids}\label{subsubsection: relation to monoids}
Another important class of blueprints are monoids with zero. Namely, a monoid $A$ with zero defines the blueprint $B=\bpgenquot A\emptyset$. A morphism $f:A_1\to A_2$ of monoids with zero is nothing else than a morphism of blueprints $f:\bpgenquot {A_1}\emptyset\to \bpgenquot {A_2}\emptyset$. This defines a full embedding $\iota_{\cM_0}:{\cM_0}\to\bp$. This justifies that we may call blueprints in the essential image of $\iota_{\cM_0}$ \emph{monoids with zero}\index{Monoid with zero} and that we identify in this case $A$ with $B=\bpgenquot A\emptyset$.

\subsubsection{Idempotent semirings and sesquiads}
\label{subsubsection: Idempotent semirings and sesquiads}

Lescot develops in \cite{Lescot09}, \cite{Lescot11} and \cite{Lescot12} an algebraic geometry for idempotent semirings, i.e.\ semirings $R$ with $a+a=a$ for all elements $a\in R$. This comes together with the usual notion of a semiring morphism. As explained in Section \ref{subsubsection: relation to rings}, we can consider (idempotent) semirings as blueprints and can adopt the notation of blueprints to this case. The semiring $\B_1=\bpgenquot{\{0,1\}}{1+1\=1}$ is initial in the category of idempotent semirings. Therefore an idempotent semiring is the same as a semiring that is a $\B_1$-algebra.

Lescot compares $\B_1$-algebras with $\Fun$-algebras and finds a series of analogies. This might find a deeper explanation by the theory of blueprints, which contains both monoids and semirings over $\B_1$ as a special case. Namely, $\B_1$ is an $\Fun$-algebra (in the category of blueprints) and the base extension functor $\blanc\otimes_\Fun\B_1$ allows us to pass from from monoids to $\B_1$-algebras.


Sesquiads are another algebraic structure, which occurs in Deitmar's paper \cite{Deitmar11a} on congruence schemes. It turns out that a sesquiad is nothing else than a cancellative blueprint, see \cite[Section 1.8]{blueprints1} for details. Also confer the remarks on congruence schemes in Section \ref{subsection: congruence schemes} of this paper.

\subsubsection{Cyclotomic field extensions}\label{subsubsection: valuation blueprints}
We give some other examples, which are of interest for the purposes of this text. The initial object in $\bp$ is the monoid $\Fun=\{0,1\}$, the so-called \emph{field with one element}\index{Field with one element}. More general, we define the \emph{cyclotomic field extension $\Funn$ of $\Fun$}\index{Field with one element!cyclomtomic extensions} as the blueprint $B=\bpquot A\cR$ where $A=\{0\}\cup\mu_n$ is the union of $0$ with a cyclic group $\mu_n=\{\zeta_n^i|i=1,\dotsc,n\}$ of order $n$ with generator $\zeta_n$ and where $\cR$ is generated by the relations $\sum_{i=0}^{n/d}\zeta_n^{di}\=0$ for every proper divisor $d$ of $n$. The associated ring of $\Funn$ is the ring $\Z[\zeta_n]$ of integers of the cyclotomic field extension $\Q[\zeta_n]$ of $\Q$ that is generated by the $n$-th roots of unity. 

Note that this breaks with the convention of $\Fun$-literature that $\Funn$ should be a monoid and its associated ring should be isomorphic to $\Z[T]/(T^n-1)$.

\subsection{Ideals and units}
\label{subsection: ideals and units}

\noindent
Let $B=\bpquot A\cR$ be a blueprint. An \emph{ideal}\index{Ideal of a blueprint} of $B$ is a subset $I$ satisfying that $IB\subset I$ and that for every additive relation of the form $\sum a_i+ c\=\sum b_j$ in $B$ with $a_i,b_j\in I$, we have $c\in I$. The relation $0\=\text{(empty sum)}$ implies that every ideal $I$ contains $0$.

If $f:B\to C$ is a blueprint morphism, then $I=f^{-1}(0)$ is an ideal of $B$ (cf.\ \cite[Prop.\ 2.14]{blueprints1}). The \emph{quotient $B/I$}\index{Quotient of a blueprint} of a blueprint $B$ by an ideal $I$ is characterized by the following universal property, cf.\ \cite[Prop.\ 2.13]{blueprints1}.

\begin{prop}
 Let $I$ be an ideal of a blueprint $B$. Then there exists a surjective blueprint morphism $f:B\to B/I$ with $I=f^{-1}(0)$ such that any other blueprint morphism $h:B\to C$ with $I\subset h^{-1}(0)$ factors into $h=g\circ f$ for a unique blueprint morphism $g:B/I\to C$.
\end{prop}

A \emph{multiplicative set}\index{Multiplicative set} is a subset $S$ of $B$ that is closed under multiplication and contains $1$. An ideal $\fp$ of $B$ is a \emph{prime ideal}\index{Prime ideal} if its complement in $B$ is a multiplicative set. A \emph{maximal ideal}\index{Maximal ideal} is an ideal that is maximal for the inclusion relation and not equal to $B$ itself. A blueprint is \emph{local}\index{Blueprint!local} if it has a unique maximal ideal. A morphism $f:B_1\to B_2$ between local blueprints is \emph{local}\index{Morphism!local} if it maps the maximal ideal of $B_1$ to the maximal ideal of $B_2$.

Let $B=\bpquot A\cR$ be a blueprint. The \emph{units of $B$}\index{Units of a blueprint} is the multiplicative subgroup $B^\times$ of all invertible elements in $A$. A \emph{blue field}\index{Blue field} is a blueprint $B=\bpquot A\cR$ with $B^\times=A-\{0\}$. For an arbitrary blueprint $B=\bpquot A\cR$, the \emph{unit field of $B$}\index{Unit field} is the blue field $B^{\px}=\bpquot{(B^\times\cup\{0\})}{\cR^\px}$ where $\cR^\px=\cR\vert_{B^\times\cup\{0\}}$ is the restriction of $\cR$ to the submonoid $B^\times\cup\{0\}$ of $A$. The unit field comes together with a canonical inclusion $u:B^\px\to B$ of blueprints. Note that $B$ is a blue field if and only if the inclusion $u:B^\px\to B$ is an isomorphism of blueprints.

If $B$ is a ring, all the above definitions specialize to the corresponding definitions for rings. The same is true for monoids with zero. The following well-known facts for rings generalize to blueprints: maximal ideals are prime ideals; if $B$ is a local blueprint, then its maximal ideal $\fm$ is the complement of $B^\times$ in $B$; an ideal $\fp$ of $B$ is a prime ideal if and only if $B/\fp$ is without \emph{zero-divisors}\index{Zero-divisor}, i.e.\ non-zero elements $a$ and $b$ with $ab=0$; an ideal $\fm$ of $B$ is a maximal ideal if and only if $B/\fm$ is a blue field.

\subsection{Blue schemes}
\label{subsection: locally blueprinted spaces}

\noindent
A \emph{blueprinted space}\index{Blueprinted space} is a topological space $X$ together with a sheaf $\cO_X$ in $\bp$. A \emph{morphism of blueprinted spaces}\index{Morphism!of blueprinted spaces} is a continuous map together with a sheaf morphism. Since the category $\bp$ contains small colimits, the stalks $\cO_{X,x}$ in points $x\in X$ exist, and a morphism of blueprinted spaces induces morphisms between stalks. A \emph{locally blueprinted space}\index{Locally blueprinted space} is a blueprinted space whose stalks $\cO_{X,x}$ are local blueprints with maximal ideal $\fm_x$ for all $x\in X$. A \emph{local morphism}\index{Morphism!local} between locally blueprinted spaces is a morphism of blueprinted spaces that induces local morphisms of blueprints between all stalks. We denote the resulting category by $\bpspaces$.

Let $x$ be a point of a locally blueprinted space $X$. We define the \emph{residue field of $x$}\index{Residue field} as the blue field $\kappa(x)=\cO_{X,x}/\fm_x$. A local morphism of locally blueprinted spaces induces morphisms between residue fields.

The \emph{spectrum of a blueprint $B$}\index{Spectrum of a blueprint} is defined analogously to the case of rings or monoids with zero: $\Spec B$ is the locally blueprinted space whose underlying set $X$ is the set of all prime ideals of $B$, endowed with the Zariski topology, and whose structure sheaf $\cO_X$ consists of localizations of $B$. For more details, see Section 3.1 of \cite{blueprints1}. A \emph{blue scheme}\index{Blue scheme} is a locally blueprinted space that is locally isomorphic to spectra of blueprints. We denote the full subcategory of $\bpspaces$ whose objects are blue schemes by $\BSch$.

\begin{ex}[Affine spaces]
 Let $B=\Fun[T_1,\dotsc,T_n]$ be the monoid $\{0\}\cup\{\,T_1^{e_1}\dotsb T_n^{e_n} \, | \, e_1,\dotsc,e_n\geq0\,\}$, which is the free blueprint in $T_1,\dotsc,T_n$ over $\Fun$. Since the associated ring $B_\Z^+$ is the polynomial ring $\Z[T_1,\dotsc,T_n]^+$, we define $\A^n_\Fun=\Spec\Fun[T_1,\dotsc,T_n]$.

 We describe the underlying topological space of $\A^n_\Fun$ for $n=1$ and $n=2$. The underlying topological space of the \emph{affine line $\A^1_{\Fun}=\Spec\Fun[T]$ over $\Fun$}\index{Affine line over $\Fun$} consists of the prime ideals $(0)=\{0\}$ and $(T)=\{0\}\cup\{T^i\}_{i>0}$, the latter one being a specialization of the former one. The \emph{affine plane $\A^2_\Fun=\Spec\Fun[S,T]$ over $\Fun$}\index{Affine plane over $\Fun$} has four points $(0)$, $(S)$, $(T)$ and $(S,T)$. More generally, the prime ideals of $\Fun[T_1,\dotsc,T_n]$ are of the form $\fp_I=(T_i)_{i\in I}$ where $I$ ranges through the subsets of $\{1,\dotsc,n\}$. The affine line and the affine plane are illustrated in Figure \ref{figure: a1,a2}. The lines between two points indicate that the point at the top is a specialization of the point at the bottom. 
 \begin{figure}[h]
  \begin{center}
   \includegraphics{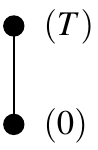} \hspace{2cm} \includegraphics{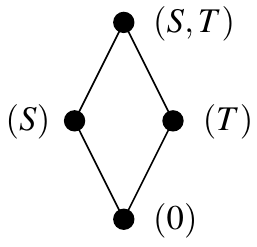}   
   \caption{The affine line and the affine plane over $\Fun$}
   \label{figure: a1,a2}
  \end{center} 
 \end{figure}
\end{ex}

\begin{ex}
 The example that initially led to the development of blueprints and blue schemes is the spectrum $\SL_{2,\Fun}=\Spec B$ of the coordinate blueprint 
 \[ 
  B \quad = \quad \bpgenquot{\Fun[T_1,\dotsc,T_4]}{T_1T_4\= T_2T_3+1}
 \]
 of $\SL_2$ over $\Fun$. Since the canonical surjection $\Fun[T_1,T_2,T_3,T_4]\to B$ defines a closed embedding $\SL_{2,\Fun}\hookrightarrow \A^4_\Fun$, the points of $\SL_{2,\Fun}$ are all of the form $\fp_I=(T_i)_{i\in I}$ where $I$ is a subset of $\{1,2,3,4\}$. If $\fp_I$ contains with $T_1$ and $T_2$ also $T_1T_4$ and $T_2T_3$ (since $\fp_I B=\fp_I$), and $T_1T_4\= T_2T_3+1$ implies that $1\in\fp_I$. This means that $\fp_I=B$, i.e.\ $\fp_I$ is not a prime ideal, if $\{1,2\}\subset I$. The same reasoning applies to subsets $I$ that contain $\{1,3\}$, $\{2,4\}$ or $\{3,4\}$. To not include any of these sets is the only restraint on $I$ for $\fp_I$ to be a prime ideal. Thus the points of $\SL_{2,\Fun}$ are as follows:
 \begin{center}
   \includegraphics{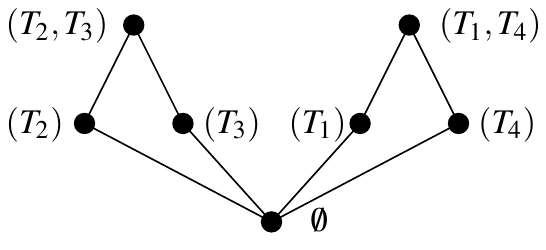}
 \end{center}

 Note that the two closed points $(T_2,T_3)$ and $(T_1,T_4)$ correspond to the diagonal torus $\tinymat \ast00\ast$ ($T_2=T_3=0$) and the antidiagonal torus $\tinymat 0\ast\ast0$ ($T_1=T_4=0$) in $\SL_2$, which are the elements of the Weyl group of $\SL_2$. This observation led to a realization of Jacques Tits' idea of Chevalley groups over $\Fun$. We will describe the $\Fun$-theory of Chevalley groups in Section \ref{section: chevalley groups over fun}.
\end{ex}

\subsubsection{Relation to usual schemes and monoidal schemes}
\label{subsubsection: Relation to usual schemes and monoidal schemes}

Since the definition of a blue scheme is formally the same as the definition of a usual scheme (in the sense of Grothendieck), the full embedding $\iota_\cR:\Rings\to \bp$ defines a full embedding of the category $\Sch^+_\Z$ of usual schemes into the category of blueprints. The full embedding $\iota_{\cM_0}:{\cM_0}\to\bp$ defines a full embedding of the category of \emph{monoidal schemes} (or \emph{$\Fun$-schemes} in the sense of Deitmar, \cite{Deitmar05}, or \emph{$\cM_0$-schemes} in the sense of Connes and Consani, \cite{Connes-Consani10a}) into the category of blue schemes. We call the objects in the corresponding essential images \emph{Grothendieck schemes}\index{Grothendieck scheme} resp.\ \emph{monoidal schemes}\index{Monoidal scheme}.

A Grothendieck scheme $X$ is characterized by the property that for all open subsets $U$ of $X$, $\cO_X(U)$ is a ring. Similarly, a monoidal scheme $X$ is characterized by the property that for all open subsets $U$ of $X$, $\cO_X(U)$ is a monoid with zero. A \emph{semiring scheme}\index{Semiring scheme} is a blue scheme $X$ such that for all open subsets $U$ of $X$, $\cO_X(U)$ is a semiring. We denote the full subcategory of semiring schemes in $\BSch$ by $\Sch_\N^+$.

\subsubsection{Base extension to semiring schemes and Grothendieck schemes}
\label{subsubsection: base extension to semirings schemes and Grothendieck schemes}

The ``base extension''-functors $(\blanc)^+$ and $(\blanc)_\Z^+$ from blueprints to semirings resp.\ rings extend to functors $(\blanc)^+:\BSch\to\Sch_\N$ resp.\ $(\blanc)_\Z^+:\BSch\to\Sch_\Z$. If $X$ is a blue scheme, then the base extensions to (semi)rings come together with the base change morphisms $X^+\to X$ and $X^+_\Z\to X$.

\subsection{Globalizations and fibre products}
\label{subsection: globalizations}
\label{subsection: properties of blue schemes}

A morphism $f:B\to C$ of blueprints induces a local morphism $f^\ast:\Spec C\to\Spec B$. A \emph{locally algebraic morphism of blue schemes}\index{Morphism!locally algebraic} is a morphism that is locally induced by morphisms between spectra of blueprints.

In contrast to usual algebraic geometry, the functor $\Spec:\bp\to\BSch$ is not a full embedding and the global section functor $\Gamma:\BSch\to\bp$ is not a retract to $\Spec$. More precisely, if $X=\Spec B$ and $\Gamma B=\Gamma(X, \cO_X)$, then the canonical morphism $B\to \Gamma B$ does not need to be an isomorphism, see Example \ref{ex: a blueprint that is not global} below. 

The blueprint $\Gamma B$ is called the \emph{globalization of $B$}\index{Globalization}, and $B$ is \emph{global}\index{Blueprint!global} if $B\to\Gamma B$ is an isomorphism. The following theorem shows that the possible discrepancy between a blueprint and its globalization does not occur on a geometric level.

\begin{thm}[{\cite[Thm.\ 3.12]{blueprints1}}]\label{thm: globalization induces an isomorphism of spectra}
 The canonical morphism $\Spec \Gamma B\to\Spec B$ is an isomorphism of blue schemes. Consequently, the globalization of a blueprint is global.
\end{thm}

\begin{ex}[A blueprint that is not global]\label{ex: a blueprint that is not global}
 Let $k_1$ and $k_2$ be two fields and consider the following blueprint $B=\bpquot A\cR$. The underlying multiplicative monoid is $A=k_1^\bullet\times k_2^\bullet$. The pre-addition $\cR$ is generated by the relations $\sum(a_i,0)\=\sum(b_j,0)$ if $\sum a_i\=\sum b_j$ in $k_1$ and $\sum(0,a_i)\=\sum(0,b_j)$ if $\sum a_i\=\sum b_j$ in $k_2$. Then $\cR$ is not an addition, e.g.\ the sum $(1,1)+(1,1)$ is not defined. The prime ideals of $B$ are $\fp_1=k_1\times\{0\}$ and $\fp_2=\{0\}\times k_2$, thus $\Spec B$ is a discrete space with two points $\fp_1$ and $\fp_2$. The stalk at $\fp_1$ is $k_1$ and the stalk at $\fp_2$ is $k_2$. Therefore the blueprint of global section is the product of $k_1$ and $k_2$, which is a ring and thus does not equal $B$. For an example of a different vein, cf.\ \cite[Ex.\ 3.8]{blueprints1}.
\end{ex}

Theorem \ref{thm: globalization induces an isomorphism of spectra} allows us to extend some basic properties of usual schemes to blue schemes.
\begin{enumerate}
 \item The affine open subschemes of a blue scheme form a basis for its topology (cf.\ \cite[Cor.\ 3.22]{blueprints1}).
 \item Morphisms of blue schemes are {locally algebraic} (cf.\ \cite[Thm.\ 3.23]{blueprints1}). 
 \item Fibre products of blue schemes exist in $\BSch$ (cf.\ \cite[Prop.\ 3.27]{blueprints1}).
\end{enumerate}
In fact, the fibre products of blue schemes are of a much simpler nature than fiber products of Grothendieck schemes (in the category of Grothendieck schemes). This has the important effect that the fibre product in $\BSch$ coincides with the fibre product in $\bpspaces$, which is absolutely not true for Grothendieck schemes and locally ringed spaces. More precisely, the following is true.

\begin{thm}\label{thm: fibre products}
 The category $\bpspaces$ has fibre products. The fibre product $X\times _S Y$ is naturally a subset of the topological product $X\toptimes Y$, and it carries the subspace topology. In the case of $S=\Spec\Funn$, it has the explicit description
 \[
  X\times_\Funn Y \qquad = \qquad \left\{ \ (x,y)\in X\toptimes Y \ \left| \ \substack{\text{there are a semifield }k\text{ and blueprint}\\ \text{morphisms } \kappa(x)\to k\text{ and }\kappa(y)\to k} \ \right.\right\}.
 \]
 If $X$, $Y$ and $S$ are blue schemes, then the fibre product $X\times_S Y$ in $\bpspaces$ coincides with the fibre product in $\BSch$. In particular, $X\times_S Y$ is a blue scheme.
\end{thm}

\begin{proof}
 This theorem follows by the same arguments used in \cite{blueprints2} to prove Proposition 1.35 and Theorem 1.38.
\end{proof}

\subsection{Projective space}
\label{subsection: projective geometry}
\label{subsection: projective space}

The generalization of the $\Proj$-construction from scheme theory to blue schemes has been published in the joint paper \cite{LL12} with Javier L\'opez Pe\~na. 

A subset $M$ of a blueprint $B$ is \emph{additively closed}\index{Additively closed subset} if for all $b\=\sum a_i$ with $a_i\in M$ also $b\in M$. In particular, an additively closed subset contains $0$. A \emph{graded blueprint}\index{Blueprint!graded} is a blueprint $B$ together with additively closed subsets $B_i$ for $i\in\N$ such that $1\in B_0$, such that for all $i,j\in\N$ and $a\in B_i$, $b\in B_j$, the product $ab$ is an element of $B_{i+j}$ and such that for every $b\in B$, there are a unique finite subset $I$ of $\N$ and unique non-zero elements $a_i\in B_i$ for every $i\in I$ such that $b\=\sum a_i$. An element of $B_\hom=\bigcup_{i\geq 0}B_i$ is called \emph{homogeneous}\index{Homogeneous element}. There is a straight forward generalization of the functor $\Proj$ from graded rings to graded blueprints. For details, see \cite[Section 2]{LL12}.

The functor $\Proj$ allows us the definition of the \emph{projective space $\P^n_B$}\index{Projective space} over a blueprint $B$. Namely, the free blueprint $C=B[T_0,\dotsc,T_n]$ over $B$ comes together with a natural grading $C_\hom=\bigcup_{i\in\N}C_i$ where the $C_i$ consist of all monomials $bT_0^{e_0}\dotsb T_n^{e_n}$ such that $e_0+\dotsb+e_n=i$ where $b\in B$. Note that $C_0=B$ and $C_\hom=C$. The projective space $\P^n_B$ is defined as $\Proj B[T_0,\dotsc,T_n]$. It comes together with a structure morphism $\P^n_B\to\Spec B$.

In case of $B=\Fun$, the projective space $\P^n_\Fun$ is the monoidal scheme that is known from $\Fun$-geometry (cf.\ \cite{Deitmar08}, \cite[Section 3.1.4]{CLS12}) and Example \ref{ex: projective line and plane} below). The topological space of $\P^n_\Fun$ is finite. Its points correspond to the homogeneous prime ideals $\fp_I=(S_i)_{i\in I}$ of $\Fun[S_0,\dotsc,S_n]$ where $I$ ranges through all proper subsets of $\{0,\dotsc,n\}$.

In case of a ring $B$, the projective space $\P^n_B$ does not coincide with the usual projective space since the free blueprint $B[S_0,\dotsc,S_n]$ is not a ring, but merely the blueprint of all monomials of the form $bS_0^{e_0}\dotsb S_n^{e_n}$ with $b\in B$. However, the associated scheme $\SP^n_B=(\P^n_B)^+$ coincides with the usual projective space over $B$, which equals $\Proj B[S_0,\dotsc,S_n]^+$.

\begin{ex}\label{ex: projective line and plane}
 The \emph{projective line $\P^1_{\Fun}=\A^1_{\Fun}\coprod_{\G_{m,\Fun}}\A^1_{\Fun}$ over $\Fun$}\index{Projective line over $\Fun$} has two closed points $[0:1]$ and $[1:0]$ and one generic point $[1:1]$. The points of the \emph{projective plane $\P^2_\Fun$ over $\Fun$}\index{Projective plane over $\Fun$} correspond to all combinations $[x_0:x_1:x_2]$ with $x_i=0$ or $1$ with exception of $x_0=x_1=x_2=0$. The projective line $\P^1_\Fun$ and the projective plane $\P^2_\Fun$ are illustrated in Figure \ref{figure: p1,p2}.
\begin{figure}[h]
 \begin{center}
  \includegraphics{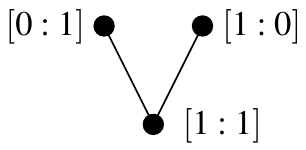} \hspace{2cm} \includegraphics{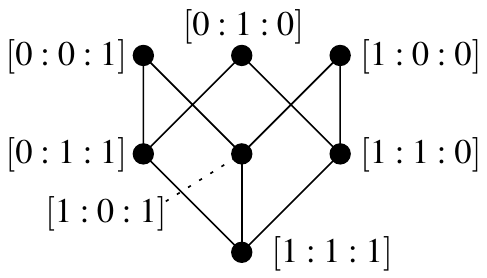}
  \caption{The projective line and the projective plane over $\Fun$}
  \label{figure: p1,p2}
 \end{center} 
\end{figure}
\end{ex}


\subsection{Closed subschemes}
\label{subsection: closed subschemes}

A blue scheme $X$ is \emph{of finite type}\index{Blue scheme!of finite type} if $X$ is compact and if for all open subsets $U$ of $X$, $\cO_X(U)$ is finitely generated as a monoid. It follows that $X$ has finitely many points.

Let $\cX$ be a scheme of finite type. By an \emph{$\Fun$-model of $\cX$}\index{F1-model of a scheme@$\Fun$-model of a scheme} we mean a blue scheme $X$ of finite type such that $X_\Z^+$ is isomorphic to $\cX$. Since a finitely generated $\Z$-algebra is, by definition, generated by a finitely generated multiplicative subset as a $\Z$-module, every scheme of finite type has an $\Fun$-model. It is, on the contrary, true that a scheme of finite type possesses a large number of $\Fun$-models.

Given a scheme $\cX$ with an $\Fun$-model $X$, we can associate to every closed subscheme $\cY$ of $\cX$ the following closed subscheme $Y$ of $X$, which is an $\Fun$-model of $\cY$. In case that $X=\Spec B$ is the spectrum of a blueprint $B=\bpquot A\cR$, and thus $\cX\simeq\Spec B_\Z^+$ is an affine scheme, we can define $Y$ as $\Spec C$ for $C=\bpquot A{\cR(Y)}$ where $\cR(Y)$ is the pre-addition that contains $\sum a_i\=\sum b_j$ whenever $\sum a_i=\sum b_j$ holds in the coordinate ring $\Gamma\cY$ of $\cY$.

Since localizations commute with additive closures, i.e.\ $(S^{-1}B)^+_\Z = S^{-1}(B^+_\Z)$ where $S$ is a multiplicative subset of $B$, the above process is compatible with the restriction to affine opens $U\subset X$. This means that given $U=\Spec(S^{-1}B)$, which is an $\Fun$-model for $\cX'=U_\Z^+$, then the $\Fun$--model $Y'$ that is associated with the closed subscheme $\cY' = \cX'\times_\cX\cY$ of $\cX'$ by the above process is the spectrum of the blueprint $S^{-1}C$. Consequently, we can associate with every closed subscheme $\cY$ of a scheme $\cX$ with an $\Fun$-model $X$ a closed subscheme $Y$ of $X$, which is an $\Fun$--model of $\cY$; namely, we apply the above process to all affine open subschemes of $\cX$ and glue them together, which is possible since additive closures commute with localizations.

In case of a projective variety, i.e.\ a closed subscheme $\cY$ of a projective space $\SP^n_\Z$, we derive the following description of the associated $\Fun$-model $Y$ in $\P^n_\Fun$ by homogeneous coordinate rings. Let $C$ be the homogeneous coordinate ring of $\cY$, which is a quotient of $\Z[S_0,\dotsc,S_n]^+$ by a homogeneous ideal $I$. Let $\cR$ be the pre-addition on $\Fun[S_0,\dotsc,S_n]$ that consists of all relations $\sum a_i\=\sum b_j$ such that $\sum a_i=\sum b_j$ in $C$. Then $B=\bpquot{\Fun[S_0,\dotsc,S_n]}{\cR}$ inherits a grading from $\Fun[S_0,\dotsc,S_n]$ and the $\Fun$-model $Y$ of $\cY$ equals $\Proj B$.


\section{Chevalley groups over $\Fun$}
\label{section: chevalley groups over fun}

A \emph{Chevalley group}\index{Chevalley group} is a split reductive group scheme, which is sometimes assumed to be simple or semisimple, depending on the literature. We refer to SGA3 (\cite{SGA3I}, \cite{SGA3II} and \cite{SGA3III}) and Conrad's lecture notes \cite{Conrad11} for details on algebraic groups and reductive group schemes. In the theory of blueprints, it is possible to define an $\Fun$-model $G$ of a Chevalley group $\cG$, i.e.\ a blue scheme $G$ of finite type over $\Fun$ whose base extension to Grothendieck schemes is isomorphic to $\cG$. However, since the group law $\mu_\cG$ of $\cG$ involves addition (unless $\cG$ is a split torus, cf.\ Example \ref{ex: algebraic torus}), it is in general not possible to descend $\mu_\cG:\cG\times\cG\to\cG$ to a locally algebraic morphism $\mu: G\times G\to G$.

The idea is to replace locally algebraic morphisms by a different notion of a morphism of blue schemes, which allows us to descend group laws and group actions. In honour of Jacques Tits who initially posted the problem of defining Chevalley groups over $\Fun$, we call these new morphisms \emph{Tits morphisms}.

A central role in the definition of Tits morphisms is played by the rank space of a blue scheme, which can be thought of as its set of $\Fun$-rational points together with a residue field for each point. This reflects the ad hoc notions of \emph{$\Fun$-rational points}\index{F1-rational points@$\Fun$-rational points} for Chevalley groups (see the beginning of Part \ref{partI}), for projective spaces and other schemes with counting polynomials (see Section \ref{section: euler cahracteristics and f1-rational points}) and for quiver Grassmannians (see Section \ref{subsection: naive set of f1-rational points}). 

Though the \emph{Tits category} $\Sch_\cT$ behaves differently from the category $\Sch_\Fun$ of blue schemes, they have some properties in common. Tits morphisms and locally algebraic morphisms agree for semiring schemes, i.e.\ $\Sch_\N^+$ is a full subcategory of both $\Sch_\Fun$ and $\Sch_\cT$. Similar to $\Sch_\Fun$, the Tits category comes together with a base extension functor $(\blanc)^+:\Sch_\cT\to\Sch_\N^+$ to semiring schemes. The other extreme of blue schemes for which Tits morphisms and locally algebraic morphisms agree are rank spaces.

One important ingredient that is new to the Tits category is the Weyl extension $\cW:\Sch_\cT\to \Sets$, which is a functor that associates with a blue scheme the underlying set of its rank space. This functor can be thought of as the scheme-theoretic substitute for the ad hoc definitions of sets of $\Fun$-rational points. In particular, it is possible to resolve the dilemma with the formula $G(\Fun)=W$ (see the beginning of Part \ref{partI}) if we substitute $G(\Fun)=\Hom(\Spec\Fun,G)$ by $\cW(G)$.

The idea of an algebraic group over $\Fun$ is made precise in the definition of a \emph{Tits-Weyl model}. Roughly speaking, a Tits-Weyl model of an affine smooth group scheme $\cG$ of finite type is a blue scheme $G$ together with a multiplication $\mu:G\times G\to G$ in $\Sch_\cT$ such that $G^+_\Z\simeq \cG$ (as a group scheme) and such that $\cW(G)$ is canonically isomorphic to the Weyl group $W$ of $\cG$ (w.r.t.\ some fixed maximal torus in $\cG$).

In the following, we will explain the concepts mentioned above and show that it is indeed applicable to Chevalley groups and many other types of algebraic groups.

\begin{ex}[Algebraic torus over $\Fun$]\label{ex: algebraic torus}
 Split tori\index{Torus} are the only connected group schemes that fit easily into any concept of $\Fun$-geometry. We examine the case $r=1$ from the perspective of blue schemes; the case of higher rank $r$ can be deduced easily from the following description.

 For a blueprint $B$, we define the \emph{multiplicative group scheme $\G_{m,B}$}\index{Multiplicative group scheme} as the spectrum of $B[T^{\pm1}]$, which is defined as the following blueprint $\bpquot A\cR$. The monoid $A$ is defined as $\{0\}\cup\{aT^i|a\in B-\{0\}, i\in \Z\}$ with the multiplication $(aT^i)(bT^j)=abT^{i+j}$ if $ab\neq0$ and $(aT^i)(bT^j)=0$ if $ab=0$. We can consider $B$ as a subset of $A$ by identifying $a\in B$ with $aT^0\in A$ (resp.\ with zero if $a=0$). Then $\cR$ is the pre-addition on $A$ that is generated by the pre-addition of $B$, considered as a set of additive relations between elements of $A$. 
 
 The multiplication $\mu:\G_{m,B}\times \G_{m,B}\to\G_{m,B}$ is given by the morphism
 \[
  \begin{array}{cccc}
   \Gamma \mu: & B[T^{\pm1}] & \longrightarrow & B[T^{\pm1}] \otimes_B B[T^{\pm1}]. \\
               & a T^i       & \longmapsto     & a T^i\otimes T^i
  \end{array}
 \]
 It is easily verified that $\G_{m,B}$ together with $\mu$ is indeed a commutative group in $\Sch_B$.

\end{ex}

\subsection{The rank space}
\label{subsection: the rank space}

A blueprint $B$ is \emph{cancellative}\index{Blueprint!cancellative} if $\sum a_i +c\=\sum b_j+c$ implies $\sum a_i\=\sum b_j$. Note that a blueprint $B$ is cancellative if and only if the canonical morphism $B\to B^+_\Z$ is injective. Recall that $\Fun$ is the monoid $\{0,1\}$ and that $\Funsq$ is the blueprint $\bpgenquot{\{0,\pm1\}}{1+(-1)\=0}$.

Let $X$ be a blue scheme and $x\in X$ a point. As in usual scheme theory, every closed subset $Z$ of $X$ comes with a natural structure of a \emph{(reduced) closed subscheme of $X$}\index{Closed subscheme}, cf.\ \cite[Section 1.4]{blueprints2}. The \emph{rank $\rk\, x$ of $x$}\index{Rank!of a point} is the dimension of the $\Q$-scheme $\barx^+_\Q$ where $\barx$ denotes the closure of $x$ in $X$ together with its natural structure as a closed subscheme. Define
\[
 r \quad = \quad \min\, \{\ \rk\, x \ | \ x\in X \ \}.
\]
For the sake of simplicity, we will make the following general hypothesis on $X$. 
\begin{enumerate}
 \item[(H)]\label{hypothesis} The blue scheme $X$ is connected and cancellative. For all $x\in X$ with $\rk\, x=r$, the closed subscheme $\barx$ of $X$ is isomorphic to either $\G_{m,\Fun}^{r}$ or $\G_{m,\Funsq}^{r}$.
\end{enumerate}
This hypothesis allows us to surpass certain technical aspects in the definition of the rank space. 

Assume that $X$ satisfies (H). Then the number $r$ is denoted by $\rk\,X$ and is called the \emph{rank of $X$}\index{Rank!of a blue scheme}. The \emph{rank space of $X$}\index{Rank space} is the blue scheme
\[
 X^\rk\quad = \quad \coprod_{\rk\, x=r}\ \barx,
\]
and it comes together with a closed immersion $\rho_X:X^\rk\to X$. 

Note that Hypothesis (H) implies that $X^\rk$ is the disjoint union of tori $\G_{m,\Fun}^r$ and $\G_{m,\Funsq}^r$. Since the underlying set of both $\G_{m,\Fun}^r$ and $\G_{m,\Funsq}^r$ is the one-point set, the underlying set of $X^\rk$ is $\cW(X)=\{x\in X|\rk\,x=r\}$. Note further that $X^{\rk,+}_\Z\simeq\coprod_{\rk\,x=r} \SG_{m,\Z}^r$. For any blue scheme $Y$, we denote by $\beta_Y:Y^+_\Z\to Y$ the base extension morphism. This means that we obtain a commutative diagram
\[
 \xymatrix@C=4pc{   X^{\rk,+}_\Z \ar[r]^{\rho^+_{X,\Z}}\ar[d]^{\beta_{X^\rk}}   &  X^{+}_\Z \ar[d]^{\beta_{X}} \\
                    X^\rk \ar[r]^{\rho_X}                                    &  X\ .\hspace{-5pt} } 
\]
It is easy to see that the above definitions coincide with the ones given in \cite{blueprints2} if we assume Hypothesis (H). 

\subsection{The Tits category and the Weyl extension}
\label{subsection: The Tits category and the Weyl extension}

A \emph{Tits morphism $\varphi:X\to Y$}\index{Tits morphism} between two blue schemes $X$ and $Y$ is a pair $\varphi=(\varphi^\rk,\varphi^+)$ of a locally algebraic morphism $\varphi^\rk: X^\rk\to Y^\rk$ and a locally algebraic morphism $\varphi^+: X^+\to Y^+$ such that the diagram
\[
 \xymatrix@C=6pc{X^{\rk,+}_\Z  \ar[r]^{\varphi^{\rk,+}_\Z} \ar[d]_{\rho_{X,\Z}^+}      & Y^{\rk,+}_\Z  \ar[d]^{\rho_{Y,\Z}^+} \\ 
           X^+_\Z  \ar[r]^{\varphi^+_\Z}                                      & Y^+_\Z  }
\]
commutes where $\varphi^+_\Z$ is the base extension of $\varphi^+$ to Grothendieck schemes and $\varphi^{\rk,+}_\Z$ is the base extension of $\varphi^\rk$ to Grothendieck schemes. We denote the category of blue schemes together with Tits morphisms by $\Sch_\cT$ and call it the \emph{Tits category}\index{Tits category}.

The Tits category comes together with two important functors. The \emph{Weyl extension $\cW:\Sch_\cT\to\Sets$}\index{Weyl extension} sends a blue scheme $X$ to the underlying set $\cW(X)=\{x\in X|\rk\,x=r\}$ of $X^\rk$ and a Tits morphism $\varphi:X\to Y$ to the underlying map $\cW(\varphi):\cW(X)\to \cW(Y)$ of the morphism $\varphi^\rk:X^\rk\to Y^\rk$. The base extension $(\blanc)^+:\Sch_\cT\to\Sch^+$ sends a blue scheme $X$ to its universal semiring scheme $X^+$ and a Tits morphism $\varphi:X\to Y$ to $\varphi^+:X^+\to Y^+$. We obtain the following diagram of ``base extension functors''
\[
 \xymatrix{\Sets & & \Sch^+_\Z & & \Sch^+_R \\
                 & &         & \Sch^+_\N \ar@{-}[ul]^{(\blanc)_\Z^+}\ar@{-}[ur]_{(\blanc)_R^+} \\
                 & & \TSch \ar@{-}[uull]^{\cW}\ar@{-}[ur]_{(\blanc)^+} }
\]
from the Tits category $\TSch$ to the category $\Sets$ of sets, to the category $\Sch^+_\Z$ of Grothendieck schemes and to the category $\Sch^+_R$ of semiring schemes over any semiring $R$.

\begin{thm}[{\cite[Thm.\ 3.8]{blueprints2}}]
 All functors appearing in the above diagram commute with finite products. Consequently, all functors send (semi)group objects to (semi)group objects.
\end{thm}

\subsection{Tits-Weyl models}
\label{subsection: Tits-Weyl models}

A \emph{Tits monoid}\index{Tits monoid} is a (not necessarily commutative) monoid in $\Sch_\cT$, i.e.\ a blue scheme $G$ together with an associative multiplication $\mu:G\times G\to G$ in $\Sch_\cT$ that has an identity $\epsilon:\Spec\Fun\to G$. We often understand the multiplication $\mu$ implicitly and refer to a Tits monoid simply by $G$. The Weyl extension $\cW(G)$ of a Tits monoid $G$ is a unital associative semigroup. The base extension $G^+$ is a (not necessarily commutative) monoid in $\Sch^+_\N$.

Given a Tits monoid $G$ satisfying (H) with multiplication $\mu$ and identity $\epsilon$, then the image of $\epsilon:\Spec\Fun\to G$ consists of a closed point $e$ of $X$. The closed reduced subscheme $\fe=\{e\}$ of $G$ is called the \emph{Weyl kernel of $G$}\index{Weyl kernel}. 

\begin{lemma}[{\cite[Lemma 3.11]{blueprints2}}]
 The multiplication $\mu$ restricts to $\fe$, and with this, $\fe$ is isomorphic to the torus $\G_{m,\Fun}^r$ as a group scheme. 
\end{lemma}

This means that $\fe^+_\Z$ is a split torus $T\simeq\SG_{m,\Z}^r$ of $\cG=G^+_\Z$, which we call the \emph{canonical torus of $\cG$ (w.r.t.\ $G$)}\index{Canonical torus}. 

If $\cG$ is an affine smooth group scheme of finite type, then we obtain a canonical morphism 
\[
 \Psi_\fe : \ G^{\rk,+}_\Z/\fe^+_\Z \quad \longrightarrow \quad W(T)
\]
where $W(T)=\Norm_\cG(T)/\Cent_\cG(T)$ is the \emph{Weyl group of $\cG$ w.r.t.\ $T$}\index{Weyl group} (cf.\ \cite[XIX.6]{SGA3III}). We say that $G$ is a \emph{Tits-Weyl model of $\cG$}\index{Tits-Weyl model} if $T$ is a maximal torus of $\cG$ (cf.\ \cite[XII.1.3]{SGA3II}) and $\Psi_\fe$ is an isomorphism. 

This definition has some immediate consequences. We review some definitions, before we state Theorem \ref{thm: properties of tits-weyl groups}. The \emph{ordinary Weyl group of $\cG$}{\index{Weyl group!ordinary Weyl group} is the underlying group $W$ of $W(T)$ . The \emph{reductive rank of $\cG$}\index{Rank!reductive} is the rank of a maximal torus of $\cG$. For a split reductive group scheme, we denote the \emph{extended Weyl group}}\index{Weyl group!extended Weyl group} or \emph{Tits group} $\Norm_\cG(T)(\Z)$ by $\widetilde W$ (cf.\ \cite{Tits66} or \cite[Section 3.3]{blueprints2}).

For a blueprint $B$, the set $G^\sT(B)$ of Tits morphisms from $\Spec B$ to $G$ inherits the structure of an associative unital semigroup. In case, $G$ has several connected components, we define the \emph{rank of $G$}\index{Rank!of a Tits-Weyl model} as the rank of the connected component of $G$ that contains the image of the unit $\epsilon:\Spec\Fun\to G$. 

\begin{thm}[{\cite[Thm.\ 3.14]{blueprints2}}]\label{thm: properties of tits-weyl groups}
 Let $\cG$ be an affine smooth group scheme of finite type. If $\cG$ has a Tits-Weyl model $G$, then the following properties hold true.
 \begin{enumerate}
  \item\label{part1} The Weyl group $\cW(G)$ is canonically isomorphic to the ordinary Weyl group $W$ of $\cG$.
  \item\label{part2} The rank of $G$ is equal to the reductive rank of $\cG$.
  \item\label{part3} The semigroup $G^\sT(\Fun)$ is canonically a subgroup of $\cW(G)$.
  \item\label{part4} If $\cG$ is a split reductive group scheme, then $G^\sT(\Funsq)$ is canonically isomorphic to the extended Weyl group $\widetilde W$ of $\cG$.
 \end{enumerate}
\end{thm}

The following theorem is proven in \cite{blueprints2} for a large class of split reductive group schemes $\cG$ and their Levi- and parabolic subgroups. Markus Reineke added an idea, which helped to extend it to all split reductive group schemes.

\begin{thm}\label{thm: chevalley groups over f1}
 \begin{enumerate}
  \item Every split reductive group scheme $\cG$ has a Tits-Weyl model $G$.
  \item Let $T$ be the canonical torus of $\cG$ and $\cM$ a Levi subgroup of $\cG$ containing $T$. Then $\cM$ has a Tits-Weyl model $M$ that comes together with a locally closed embedding $M\to G$ of Tits-monoids that is a Tits morphism.
  \item Let $\cP$ a parabolic subgroup of $\cG$ containing $T$. Then $\cP$ has a Tits-Weyl model $P$ that comes together with a locally closed embedding $P\to G$ of Tits-monoids that is a Tits morphism.
  \item Let $\cU$ be the unipotent radical of a parabolic subgroup $\cP$ of $\GL_{n,\Z}$ that contains the diagonal torus $T$. Then $\cU$, $\cP$ and $\GL_{n,\Z}$ have respective Tits-Weyl models $U$, $P$ and $\GL_{n,\Fun}$, together with locally closed embeddings $U\to P\to \GL_{n,\Fun}$ of Tits-monoids that are Tits morphisms and such that $T$ is the canonical torus of $\cP$ and $\GL_{n,\Z}$.
 \end{enumerate}
\end{thm}

\begin{rem}
 All Tits-Weyl models of the above theorem are constructed in terms of matrix coefficients of linear representations of the algebraic group in question. In general, different linear representations (of an appropriate nature) lead to different Tits-Weyl models. It might be possible to classify the Tits-Weyl models of a (reductive) algebraic group in terms of its representation theory. 
\end{rem}

\subsection{Total positivity}
\label{subsection: total positivity}

 Matrix coefficients of particular importance are the generalized minors of Fomin and Zelevinsky (see \cite{Fomin-Zelevinsky99}) and elements of Lustzig's canonical basis (see \cite{Lusztig93}). There is a link to Tits-Weyl models, but this is far from being well-understood. In the example of $\SL_n$, however, we can make the connection to generalized minors precise. 

 Let $I$ and $J$ be subsets of $\{1,\dotsc,n\}$ of the same cardinality $k$. Then we can consider the matrix minors
 \[
  \Delta_{I,J}(A) \quad = \quad \det(a_{i,j})_{i\in I, j\in J}
 \]
 of an $n\times n$-matrix $A=(a_{i,j})_{i,j=1,\dotsc,n}$. Following Fomin and Zelevinsky, we say that a real matrix $A$ of determinant $1$ is \emph{totally non-negative}\index{Totally non-negative matrix} if $\Delta_{I,J}(A)\geq0$ for all matrix minors $\Delta_{I,J}$, cf.\ \cite{Fomin-Zelevinsky99}.

 Let $\Fun[\Delta_{I,J}]$ be the free monoid with zero that is generated by all matrix minors $\Delta_{I,J}$ for varying $k\in\{1,\dotsc, n-1\}$ and subsets $I,J\subset\{1,\dotsc,n\}$ of cardinality $k$. Let $\cR$ be the pre-addition of $\Fun[\Delta_{I,J}]$ that consists of all additive relations $\sum a_i\=\sum b_j$ between generalized minors that hold in the coordinate ring $\Z[\SL_n]$ of $\SL_{n}$ over $\Z$. Define the blueprint
 \[
  \Fun[SL_{n}] \quad = \quad \bpquot{\Fun[\Delta_{I,J}]}{\cR}.
 \]
 and the blue scheme $\SL_{n,\Fun}=\Spec \Fun[SL_{n}]$. The following theorem is due to Javier L\'opez Pe\~na, Markus Reineke and the author.

\begin{thm}
 The blue scheme $\SL_{n,\Fun}$ has the unique structure of a Tits-Weyl model of $\SL_{n,\Z}$. It satisfies that $\SL_{n,\Fun}(\R_{\geq 0})$ is the semigroup of all totally non-negative matrices.
\end{thm}

\begin{rem}
 The Tits-Weyl model $\SL_{n,\Fun}$ yields the notion of matrices $A \in \SL_n(R)$ with coefficients in any semiring $R$. This might be of particular interest in the case of idempotent semirings as the tropical real numbers.
\end{rem}


\section{Buildings}
\label{section: buildings}

Let $\cG$ be a simple Chevalley group of type $\tau$. Then the group $\cG(\F_q)$ of $\F_q$-rational points act on the spherical building $\cB_q$\index{Building} of $\cG$ over $\F_q$. Following Jacques Tits' ideas from \cite{Tits56}, the apartments\index{Apartment} of $\cB_q$, which are Coxeter complexes of type $\tau$, should have an interpretation as the building $\cB_1$ of $\cG$ over $\Fun$. This building $\cB_1$ should come together with an action of the Weyl group, which can be thought of as the Chevalley group of type $\tau$ over $\Fun$.

A first observation is the following. It hints that the combinatorial nature of monoidal schemes has something to do with Coxeter complexes. As usual, we draw the topological space of blue schemes of finite type over $\Fun$ as a set of points together with vertical lines that indicate that the point at the top is the specialization of the point at the bottom. The topological space of $\P^2_\Fun$ without its generic point (cf. Example \ref{ex: projective line and plane}) can be considered as the Coxeter complex of type $A_2$ as illustrated in Figure \ref{figure: p2 and the coxeter complex of type a2}.
\begin{figure}[h]
 \begin{center}
\includegraphics{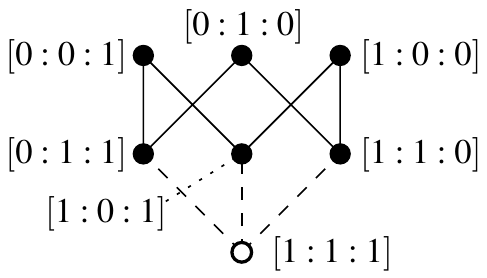} \hspace{2cm}\includegraphics{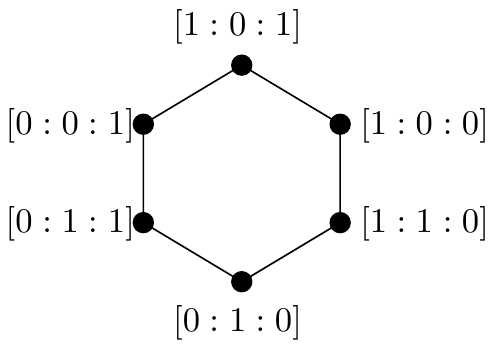}
  \caption{The projective plane over $\Fun$ and the Coxeter complex of type $A_2$}
  \label{figure: p2 and the coxeter complex of type a2}
 \end{center} 
\end{figure}

Note that the Weyl group $W$ of type $A_2$, which is the permutation group $S_3$ on three elements, acts by permuting the coordinates of the tuples $[x_0:x_1:x_2]$. For more explanations on this philosophy, see \cite{Thas13}.

\subsection{A first incidence geometry}
\label{subsection: A first incidence geometry associated with blue schemes}

The previous example extends to a correspondence of $\P^n_\Fun$ (without its generic point) and the Coxeter complex of type $A_n$ for all $n\geq1$, see Example \ref{ex: projective spaces and coxeter complexes of type a_n} below. Also Coxeter complexes of other types can be identified with certain subcomplexes of complexes associated with blue schemes.

Let $X$ be a blue scheme. We define a partial ordering of its points by $x<y$ if and only if $x\in\bary$, but $x\neq y$ where $\bary$ denotes the closure of $y$ in $X$. We define 
\[
 \tilde\Delta^n(X) \ = \ \{ \ (x_0,\dotsc,x_r)\in X^{r+1} \ | \ x_0<\dotsc< x_r \ \}. 
\]
Together with the obvious face relations, this yields a simplicial complex 
\[
 \tilde\Delta(X) \quad = \quad \bigcup_{r\geq0} \tilde\Delta^r(X).
\]
The \emph{type of a simplex $(x_0,\dotsc,x_r)$}\index{Type of a simplex} is the tuple $(\rk\, x_0,\dotsc, \rk\, x_r)$ of the ranks of the points $x_l$. 

\begin{ex}[Projective spaces and Coxeter complexes of type $A_n$]\label{ex: projective spaces and coxeter complexes of type a_n}
 If $X=\P^n_\Fun$, then the simplices of $\tilde\Delta(\P^n_\Fun)$ are chains of proper homogeneous prime ideals $\fp^{I_r}\subsetneq\dotsc\subsetneq\fp^{I_0}$ of $\Fun[X_0,\dotsc, X_n]$ where the $I_l$ are non-zero subsets of $\un=\{0,\dotsc,n\}$ and $\fp^{I_l}=(X_i)_{i\in \un-I_l}$. Such a chain of homogeneous ideals can be identified with the chain of subsets $I_0 \subsetneq\dotsc\subsetneq I_r$ of non-zero subsets of $\un$. The subcomplex of $\tilde\Delta(\P^n_\Fun)$ that consists of all chains of proper \emph{non-zero} ideals  $\fp^{I_r}\subsetneq\dotsc\subsetneq\fp^{I_0}$ corresponds to the Coxeter complex $\cC$ of type $A_n$, whose simplices are chains of proper non-empty subsets $I_0 \subsetneq\dotsc\subsetneq I_r$. 

 Let $X$ be the closed subscheme of $\P^n_\Fun$ that is supported on the complement of the generic point of $\P^n_\Fun$, which is a closed subset. Then $\tilde\Delta(X)$ consists of all chains of proper non-zero ideals, and therefore the above reasoning establishes an identification of $\tilde\Delta(X)$ with $\cC$.

 The type of a chain $\fp^{I_r}\subsetneq\dotsc\subsetneq\fp^{I_0}$ is $(\#I_0,\dotsc,\#I_r)$, which equals the type of the simplex $I_0 \subsetneq\dotsc\subsetneq I_r$ of the Coxeter complex $\cC$. This means that the identification of $\tilde\Delta(X)$ with $\cC$ is type preserving.
\end{ex}

\subsection{Representations of Tits-Weyl models}
\label{subsection: Representation of Tits-Weyl models}

Let $\cG$ be a simple Chevalley group of type $\tau$ and $G$ a Tits-Weyl model of $\cG$. A \emph{unitary action of $G$ on the projective space $\P^n_\Fun$} is a Tits morphism $\theta:G\times \P^n_\Fun\to\P^n_\Fun$ such that the diagrams
\[
 \xymatrix@C=4pc{G\times G\times {\P^n_\Fun} \ar[r]^{(\mu,\id_{\P^n_\Fun})} \ar[d]_{(\id_G,\theta)}  & G\times {\P^n_\Fun} \ar[d]^{\theta} \\ G\times {\P^n_\Fun} \ar[r]_{\theta} & {\P^n_\Fun}} \qquad \text{and} \qquad 
 \xymatrix@C=4pc{\Spec\Fun\times {\P^n_\Fun} \ar[d]_{(\epsilon,\id_{\P^n_\Fun})} \ar@{=}[r] & {\P^n_\Fun} \\ G\times {\P^n_\Fun}\ar[ru]_\theta }
\]
commute where $\mu:G\times G\to G$ is the group law and $\epsilon:\Spec\Fun\to G$ is the identity of $G$. We call such an action a \emph{representation of $G$ on $\P^n_\Fun$}\index{Representation of a Tits-Weyl model}.

Applying the Weyl extension functor $\cW$, we obtain an action 
\[
 \theta^\rk: \ W\times \cW(\P^n_\Fun) \quad \longrightarrow \quad \cW(\P^n_\Fun)
\] 
of the Weyl group $W=\cW(G)$ of $G$ on the set $\cW(\P^n_\Fun)$ of closed points of $\P^n_\Fun$, which can be identified with $\un=\{0,\dotsc,n\}$ as explained in Example \ref{ex: projective spaces and coxeter complexes of type a_n}. This extends uniquely to an action of $W$ on the underlying topological space of $\P^n_\Fun$ and on the simplicial complex $\tilde\Delta(\P^n_\Fun)$ by permuting the indices, which are elements of $\un$.

\subsection{Coxeter complexes of the classical types} 
\label{subsection: coxeter complexes of the classical types}

We describe the Coxeter complexes $\cC$ of the four classical types $A_n$, $B_n$, $C_n$ and $D_n$, together the action of the Weyl group of the classical Chevalley group $\cG$ of the corresponding type on $\cC$, in terms of representations of Tits-Weyl models $G$ of $\cG$. Similarly, we obtain a description of the spherical building $\cB_q$ over $\F_q$ together with the action of $\cG(\F_q)$ from this representation of $G$---unless $\cG$ is of type $D_n$.

\subsubsection*{Type $A_n$}
\label{subsubsection: type a_n}

Let $\cG=\SL_{n,\Z}$ and $G$ the Tits-Weyl model that is defined by the standard matrix representation of $\SL_n$, cf.\ Section 4.1 of \cite{blueprints2}. Then $G$ comes with a natural representation on $\P^n_\Fun$.

We have already seen in Example \ref{ex: projective spaces and coxeter complexes of type a_n} that $\tilde\Delta(P^n_\Fun-\{\eta\})$ is the Coxeter complex $\cC$ of type $A_n$ where $\eta$ denotes the generic point of $\P^n_\Fun$. Note that the Weyl group $W=S_{n+1}$ of type $A_n$ acts on $\cC$ by permuting the indices.

The spherical building $\cB_q$ of type $A_n$ over $\F_q$ can be described in terms of flags of proper non-zero subspaces of $\F_q^{n+1}$, in analogy the Coxeter complex $\cC$. A linear proper non-zero subspace of $\F_q^{n+1}$ corresponds to a proper linear subvariety of $\SP^n_{\F_q}$, which means that we can identify every simplex of $\cB_q$ with a simplex in $\tilde\Delta(\SP^n_{\F_q})$. Let $\Delta_q$ be the $n$-simplex in $\tilde\Delta(\SP^n_{\F_q})$ that is given by the chain $x_0<\dotsc< x_n$ where $x_l$ is the generic point of the linear subvariety of $\SP^n_{\F_q}$ that is defined by the equations $X_0=\dotsc=X_l=0$.

Knowing that $G(\F_q)$ acts transitively on $\cB_q$, we obtain the following result.

\begin{thm}
 The complex $\cC=\tilde\Delta(P^n_\Fun-\{\eta\})$ is the Coxeter complex of type $A_n$ and the representation of $G$ on $\P^n_\Fun$ yields the usual action of the Weyl group $W=\cW(G)$ on $\cC$. Taking $\F_q$-rational points yields an action of $\SL_n(\F_q)$ on the complex $\tilde\Delta(\SP^n_{\F_q})$ such that the orbit of $\Delta_q$ is the spherical building $\cB_q$ of type $A_n$ over $\F_q$ and such that the restricted action of $\SL_n(\F_q)$ is the usual action on $\cB_q$.
\end{thm}

\subsubsection*{Type $B_n$ and $C_n$}
\label{subsubsection: type b_n and c_n}

Let $\cG$ be a split form of $\SO(2n+1)$ or $\Sp(2n)$, which is a Chevalley group of type $B_n$ resp.\ $C_n$. Let $m$ be $2n+1$ resp.\ $2n$ and let $G$ be the Tits-Weyl model that is associated with the standard representation of $\cG$ as $m\times m$-matrices (see Sections 4.3.3 and 4.3.4 of \cite{blueprints2} for an explicit description of $\cG$ and $G$). Then $G$ comes together with a representation on $\P^{m-1}_\Fun$. This induces an action of the Weyl group $W=\cW(G)$ on $\tilde\Delta(\P^{m-1}_\Fun)$ and an action of $G(\F_q)$ on $\tilde\Delta(\SP^{m-1}_{\F_q})$.

Note that the representation of $G$ on $\P^{m-1}_\Fun$ comes from an action of $\cG$ on $\SA^m_\Z$, which preserves a certain quadratic form. Let $0\subsetneq V_0\subsetneq\dotsc\subsetneq V_{n-1}$ be a maximal flag of isotropic subspaces of $\SA^m_\Z$ and let $x_0,\dotsc,x_{n-1}$ denote the generic points of the corresponding linear subvarieties of $\SP^{m-1}_\Z$. The quadratic form and the subspaces $V_0,\dotsc,V_{n-1}$ can be chosen such that the linear subvarieties $\overline{x_0},\dotsc,\overline{x_n}$ are defined over $\Fun$, which defines the corresponding points $x_{0,1},\dotsc,x_{n,1}$ in $\P_\Fun^{m-1}$. We denote by $\Delta_1$ the $(n-1)$-simplex $(x_{0,1}<\dotsc<x_{n-1,1})$ in $\tilde\Delta(\P^{m-1}_{\Fun})$. For a prime power $q$, we denote by $\Delta_q$ the $(n-1)$-simplex $(x_{0,q}<\dotsc<x_{n-1,q})$ in $\tilde\Delta(\SP^{m-1}_{\F_q})$ where $x_{l,q}$ is the generic point of $\overline{x_l}\times_\Z\F_q$ in $\SP^{m-1}_{\F_q}$. 

Then we obtain the following result.

\begin{thm}
 The orbit $\cC=W.\Delta_1$ is isomorphic to the Coxeter complex of type $B_n$ resp.\ $C_n$, and the action of $W$ on $\cC$ coincides with the usual action on the Coxeter complex. The orbit $\cB_q=G(\F_q).\Delta_q$ is isomorphic to the spherical building of type $B_n$ resp.\ $C_n$, and the action of $G(\F_q)$ on $\cB_q$ coincides with the usual action on the spherical building.
\end{thm}

\subsubsection*{Type $D_n$}
\label{subsubsection: type d_n}

Let $\cG$ be a split form of $\SO(2n)$, which is a Chevalley group of type $D_n$. Let $m=2n$ and let $G$ be the Tits-Weyl model that is associated with the standard representation of $\cG$ as $m\times m$-matrices (see Section 4.3.4 of \cite{blueprints2} for a explicit description of $\cG$ and $G$).  Then $G$ comes together with a representation on $\P^{m-1}_\Fun$. This induces an action of the Weyl group $W=\cW(G)$ on $\tilde\Delta(\P^{m-1}_\Fun)$ and an action of $G(\F_q)$ on $\tilde\Delta(\SP^{m-1}_{\F_q})$.

We define the simplices $\Delta_1$ and $\Delta_q$ analogously as for the types $B_n$ and $C_n$, and obtain the following result.

\begin{thm}
 The orbit $\cC=W.\Delta_1$ is isomorphic to the Coxeter complex of type $B_n$ resp.\ $C_n$, and the action of $W$ on $\cC$ coincides with the usual action on the Coxeter complex. 
\end{thm}

\begin{rem}[The problem with the oriflamme construction]\label{rem: problem with the oriflamme construction}
 The orbit $G(\F_q).\Delta_q$ is a building, but this building is not thick and therefore not isomorphic to the spherical building $\cB_q$ of type $D_n$, cf.\ the introducing remarks of Chapter 11 in \cite{Garrett97}. This problem is usually solved by the ``oriflamme construction'' in which one considers the complex associated with pairs of flags $0\subsetneq V_0\subsetneq \dotsb\subsetneq V_{n-3}\subsetneq V_{n-1}$ and $0\subsetneq V_0\subsetneq \dotsb\subsetneq V_{n-3}\subsetneq V_{n-1}'$.
\end{rem}

\subsection{An extended incidence geometry}
\label{subsection: An extended incidence geometry}

The failure of realizing buildings of type $D_n$ as a natural subcomplex of $\tilde\Delta(\SP^{2n-1}_{\F_q})$ leads to the following extension of $\tilde\Delta(X)$ to a simplicial complex $\Delta(X)$.

As before, we consider a blue scheme $X$ as a poset given by its underlying set together with the partial order as defined by $x\leq y$ if and only if $x\in\bary$. We define an \emph{$n$-simplex}\index{Simplex} as the poset $\Delta^n$ of all non-empty subsets of $\un$. This is the same as the poset associated with $\P^n_\Fun$.

Let $X$ be a blue scheme. An \emph{$n$-simplex in $X$}\index{Simplex in a blue scheme} is an order-preserving map $\delta: \Delta^n \to X$. Let $\Delta^n(X)$ be the collection of $n$-simplices in $X$. Then
\[ \Delta(X) \ = \ \bigcup_{n\geq 0} \Delta^n(X) \]
together with the obvious face relation is a simplicial complex. 

Note that a chain $x_0<\dotsb<x_n$, together with all its subchains, can be identified with the $n$-simplex $\delta: \Delta^n\to X$ in $X$ that sends a non-zero subset $I$ of $\un$ to $\sup\{x_i\}_{i\in I}=x_{\max I}$. This shows that $\tilde\Delta(X)$ is a subcomplex of $\Delta(X)$. This simplicial complex behaves well with respect to linear representations of Chevalley groups.

\begin{prop}
 Let $\cG$ be a Chevalley group with Tits-Weyl model $G$ and $G\times \P^n_\Fun\to \P^n_\Fun$ a representation of $G$. Then the action of the Weyl group $W=\cW(G)$ on $\cW(\P^n_\Fun)$ extends uniquely to an action of $W$ on $\Delta(\P^n_\Fun)$.
\end{prop}

\subsection{The oriflamme construction}
\label{subsection: the oriflamme construction}

A pair of chains $x_0<\dotsb< x_{n-3}< x_{n-1}$ and $x_0< \dotsb< x_{n-3}< x'_{n-1}$ can be considered as the $n$-simplex $\delta: \Delta^n\to X$ that sends a non-zero subset $I$ of $\un$ to $\sup\{x_i\}_{i\in I}$. 

This allows us to carry out the oriflamme construction for type $D_n$ inside a subcomplex of $\Delta(\P^{2n-1})$. Let $G$ be the Tits-Weyl model of a split form of $\SO(2n)$ from Section \ref{subsection: coxeter complexes of the classical types}, which comes together with a representation on $\P^{m-1}_\Fun$ where $m=2n$. For $q$ equal to $1$ or a prime power, let $\delta_q:\Delta^n\to\SP^{m-1}_{\F_q}$ be the simplex that corresponds to the pair of chains $x_0<\dotsb< x_{n-3}< x_{n-1}$ and $x_0< \dotsb< x_{n-3}< x'_{n-1}$ as considered in the oriflamme construction, cf.\ Remark \ref{rem: problem with the oriflamme construction}. 

\begin{thm}
 The orbit $\cC=W.\delta_1$ is isomorphic to the Coxeter complex of type $D_n$, and the action of $W$ on $\cC$ coincides with the usual action on the Coxeter complex. The orbit $\cB_q=G(\F_q).\delta_q$ is isomorphic to the spherical building of type $D_n$, and the action of $G(\F_q)$ on $\cB_q$ coincides with the usual action on the spherical building.
\end{thm}

\subsection{Coxeter complexes of exceptional type}
\label{subsection: coxeter complexes of exceptional type}

Coxeter complexes of exceptional and classical type appear naturally as subcomplexes of the adjoint representation. Let $\cG$ be a Chevalley group of adjoint type $\tau$, of rank $r$ and of dimension $n=r+2m$. Let $\fg$ be its Lie algebra. Then a choice of fundamental roots $\Pi$ in a root system $\Phi$ for $\fg$ yields a Chevalley basis $\Psi=\Pi\cup\Phi$ for $\fg$. This determines a Tits-Weyl model $G$ of $\cG$ as explained in Section 4.4 of \cite{blueprints2}. Define
\[
 \P[\Psi] \quad = \quad \Proj \bigl( \bpquot{\Fun[X_\alpha,Y_\beta]_{\alpha,\beta\in\Pi}}\cR \bigr)
\]
where $\cR$ is generated by the relations $\sum Y_{\alpha_i}\=\sum Y_{\beta_j}$ whenever we have a relation $\sum \alpha_i=\sum \beta_j$ for roots $\alpha_i$ and $\beta_j$. Then $\P[\Psi]$ is an $\Fun$-model of the projective space $\SP[\fg]\simeq\SP^{n-1}_\Z$ that contains ``a few points more'' than $\P^{n-1}_\Fun$.

\begin{lemma}
 The adjoint representation of $\cG$ on $\SP[\fg]$ descends to a representation of $G$ on $\P[\Psi]$. The associated action of the Weyl group $W=\cW(G)$ on $\cW(\P[\Psi])$ extends uniquely to an action of $W$ on $\Delta(\P[\Psi])$.
\end{lemma}

The characterization of simplices as cosets of parabolic subgroups of the Coxeter group leads to the following result.

\begin{thm}
 There is an $(r-1)$-simplex $\delta_1$ in $\P[\Psi]$ such that the orbit $W.{\delta_1}$ in $\Delta(\P[\Psi])$ under the Weyl group $W=\cW(G)$ is a Coxeter complex of type $\tau$. The action of $W$ on $W.\delta_1$ coincides with the usual action on the Coxeter complex.
\end{thm}

Presumably, spherical buildings appear as subcomplexes of $\Delta(\SP[\fg]_{\F_q})$. More precisely, I expect the following.

\begin{conj}
 There is an $(r-1)$-simplex $\delta_{q}$ in $\SP[\fg]_{\F_q}$ such that the orbit $G(\F_q).{\delta_{q}}$ in $\Delta(\SP[\fg]_{\F_q})$ is the spherical building of type $\tau$ over $\F_q$. The action of $G(\F_q)$ on $G(\F_q).\delta_q$ coincides with the usual action on the building.
\end{conj}


\section{Euler characteristics and $\Fun$-rational points}
\label{section: euler cahracteristics and f1-rational points}

One aspect of the philosophy of $\Fun$-geometry is that the number of $\Fun$-rational points of an $\Fun$-scheme $X_\Fun$ of finite type should equal the Euler characteristic
\[
 \chi(X_\C) \quad = \quad \sum_{i=0}^{2\dim X} \ (-1)^i\, b_i
\]
where $b_i=\dim H^i(X_\C,\C)$ are the Betti numbers of the complex scheme $X_\C$. This idea is deduced from the following thought.

Let $X$ be a smooth projective scheme such that there is a polynomial $N(q) \in \Z[q]$ that counts $\F_q$-rational points, i.e.\ $\# X(\F_q)=N(q)$ for every prime power $q$. As a consequence of the comparison theorem for singular and $l$-adic cohomology and Deligne's proof of the Weil conjectures, we know that the counting polynomial is of the form
\[
 N(q)\quad = \quad \sum_{i=0}^n b_{2i}\ q^i 
\]
and that $b_i=0$ if $i$ is odd (cf.\ \cite{Kurokawa05} and \cite{Deitmar07}). Thus $\chi=\sum_{i=0}^n b_{2i}$ is the Euler characteristic of $X_\C$ in this case (cf.\ \cite{Soule04}). This equals $N(1)$, which has the interpretation as the number of $\Fun$-rational points of an $\Fun$-model $X_\Fun$ of $X$.

The hope is that a rigorous theory of $\Fun$-models of schemes together with the result that the number of its $\Fun$-rational points equals the Euler characteristic of the complex scheme will give access to computational methods and theoretical results.

There are certain questions that one faces at this point. First of all, there are varieties with a negative Euler characteristic, which cannot be the cardinality of a set. Therefore, we can only hope for the relation $\chi(X_\C)=\# X(\Fun)$ for schemes with positive Euler characteristic. For other varieties one might hope for results if one replaces the number $\# X(\Fun)$ by a weighted sum $\sum n_p$ over all $\Fun$-rational points $p$ where $n_p$ can also be negative.

Another problem is that in general, it is not clear which should be the right $\Fun$-model of a scheme $X$ in order to count its $\Fun$-rational points. It is further not clear what the right notion of an $\Fun$-rational point should be. The direct generalization as the number of homomorphism $\Spec\Fun\to X$ would produce the wrong outcome, e.g.\ $\Hom(\Spec\Fun,\P^n_\Fun)$ has $2^{n+1}-1$ elements, instead of $N(1)=n+1$ where $N(q)=q^n+\dotsc+q^0$ is the counting polynomial of $\SP^n_\Z$.

We will give a partial answer to these questions in the following.

\subsection{Quiver Grassmannians}
\label{subsection: quiver grassmannians}

Caldero and Chapoton describe in \cite{Caldero-Chapoton06} a formula that expresses cluster variables as elementary functions in the Euler characteristics of quiver Grassmannians. This approach to cluster algebras (see Fomin and Zelevinsky's paper \cite{Fomin-Zelevinsky02}) and canonical bases (see Lusztig's book \cite{Lusztig93}) caused an active interest in the Euler characteristics of quiver Grassmannians in recent years.

There is a mutual interest between quiver Grassmannians and $\Fun$-geometry. On the one side, $\Fun$-geometry searches for structure that rigidifies schemes in order to descend them to $\Fun$. It turns out that quiver Grassmannians come with such a rigidifying structure. On the other side, a successful interpretation of the Euler characteristic of a complex scheme in terms of its $\Fun$-rational points opens access to new methods to compute Euler characteristics.

Note that Reineke proves in \cite{Reineke12} that every projective scheme can be realized as a quiver Grassmannian. Therefore quiver Grassmannians are suited as a device to define $\Fun$-models of projective schemes. We will explain how this works in the following.

A \emph{quiver}\index{Quiver} is a finite directed graph $Q$. We denote the vertex set by $Q_0$, the set of arrows by $Q_1$, the source map by $s:Q_1\to Q_0$ and the target map by $t:Q_1\to Q_0$. A \emph{(complex) representation of $Q$}\index{Quiver representation} is a collection of complex vector spaces $M_i$ for each vertex $i\in Q_0$ together with a collection of $\C$-linear maps $M_\alpha:M_i\to M_j$ for each arrow $\alpha:i\to j$ in $Q_1$. We write $M=\bigl((M_i),(M_\alpha)\bigr)$ for a representation of $Q$. The \emph{dimension vector $\udim M$ of $M$}\index{Dimension vector} is the tuple $\ud=(d_i)$ where $d_i$ is the dimension of $M_i$. A \emph{subrepresentation of $M$}\index{Quiver representation!subrepresentation} is a collection of $\C$-linear subspaces $N_i$ of $M_i$ for each vertex $i\in Q_0$ such that the linear maps $M_\alpha:M_i\to M_j$ restrict to linear maps $N_i\to N_j$ for all arrows $\alpha:i\to j$.

Let $\ue=(e_i)$ be a tuple of non-negative integers with $e_i\leq d_i$ for all $i\in Q_0$. Then there is a complex scheme $\Gr_\ue(M)_\C$ whose set of $\C$-rational points corresponds to the set
\[
 \Gr_\ue(M,\C) \quad = \quad \bigl\{ \ N\subset M \text{ subrepresentation}  \ \bigl| \ \udim N = \ue    \ \bigr\}.
\]
In words, the scheme theoretic definition of $\Gr_\ue(M)$ is as the fibre at $M$ of the universal $Q$-Grassmannian of $\ue$-dimensional subrepresentations that fibres over the moduli space of representations of $Q$. For a precise definition, see Sections 2.2 and 2.3 in \cite{Cerulli-Feigin-Reineke11}.

The choice of a basis $\cB_i$ for each vector space $M_i$ defines a closed embedding
\[
 \Gr_\ue(M)_\C \quad \longrightarrow \quad \prod_{i\in Q_0} \ \Gr(e_i,d_i)_\C
\]
of the quiver Grassmannian into a product of usual Grassmannians. Therefore $\Gr_\ue(M)$ is a projective complex scheme.

\subsection{The naive set of $\Fun$-rational points}
\label{subsection: naive set of f1-rational points}

We apply the philosophy of $\Fun$-geometry that the number of $\Fun$-rational points of a projective scheme $X$ should equal its Euler-characteristic to quiver Grassmannians. This idea was explained to me by Giovanni Cerulli Irelli and Gr\'egoire Dupont. 

We recall the notion of a quiver representation over $\Fun$ from Matt Szczesny's paper \cite{Szczesny12a}. An \emph{$\Fun$-representation of a quiver $Q$}\index{Quiver representation!$\Fun$-representation} is a collection of pointed sets $S_i$ for each $i\in Q_0$ together with base point preserving maps $S_\alpha:S_i\to S_j$ whose fibres $S^{-1}_\alpha(a)$ have at most one element for $a\in S_j-\{\ast\}$.

An $\Fun$-representation $S=\bigl((S_i),(S_\alpha)\bigr)$ defines a complex representation $M$ of $Q$ as follows. Define $M_i$ as the complex vector space with basis $\cB_i=S_i-\{\ast\}$ and identify $\ast$ with $0\in M_i$. Then maps $S_\alpha:S_i\to S_j$ define $\C$-linear maps $M_\alpha:M_i\to M_j$ by linear extension. A matrix is called \emph{monomial}\index{Monomial matrix} if it contains at most one non-zero entry in each row and in each column. Note that in the basis $\cB=\bigcup\cB_i$ the linear maps are described by monomial matrices with coefficients in $\{0,1\}$. We say that a complex representation $M$ of $Q$ together with a basis $\cB=\bigcup\cB_i$ is \emph{defined by $\Fun$-coefficients}\index{Quiver representation!defined by $\Fun$-coefficients} if all linear maps $M_\alpha$ are described by monomial matrices with coefficients in $\{0,1\}$.

Given a complex representation $M$ with basis $\cB=\bigcup\cB_i$, we define the \emph{naive set of $\Fun$-rational points of $\Gr_\ue(M)$}\index{F1-rational points@$\Fun$-rational points} as the subset 
\[
 \Gr_\ue^*(M,\Fun) \quad = \quad \left\{ \ N\in\Gr_\ue(M,\C) \ \left| \begin{array}{c} N_i\text{ is spanned by }\cB_i\cap N\text{ for all }i\in Q_0 \\ \text{and }M_\alpha(\cB_i\cap N)\subset\cB_j \text{ for all }\alpha:i\to j\end{array} \right.\right\}.
\]

Note that this definition depends on the choice of basis $\cB$ and, in general, the cardinality of $\Gr^*_\ue(M,\Fun)$ will not coincide with the Euler characteristic of $\Gr_\ue(M)$. For instance for the representation 
\[
 \xymatrix{M: & \C^2 \ar[r]^{\tinymat 2002 } & \C^2}
\]
of the quiver $1\to2$ and dimension vector $\ue=(1,1)$, the set $\Gr_\ue^*(M,\Fun)$ is empty, while the Euler characteristic of $\Gr_\ue(M,\C)\simeq \P^1(\C)$ is $2$. In the following section, we define a blue scheme $\Gr_\ue(M)_\Fun$ whose Weyl extension $\cW(\Gr_\ue(M)_\Fun)$ is better suited to extract information about the Euler characteristic of $\Gr_\ue(M)$.

\subsection{$\Fun$-models of quiver Grassmannians}
\label{subsection: f1-models of quiver grassmannians}

The Grassmannian $\Gr(e,d)$ of $e$-subspaces in $d$-space is a closed subscheme of $\SP^m_\Z$ with $m=\binom de-1$, which is defined by the Pl\"ucker relations. Let $I$ be the vanishing ideal of $\Gr(e,d)$ in $\SP^m$, then $\cR=\{\sum a_i\=\sum b_j|\sum a_i-\sum b_j\in I\}$ is a pre-addition for the coordinate blueprint $\Fun[x_0,\dotsc,x_m]$ of $\P^m_\Fun$. As explained in Section \ref{subsection: closed subschemes},
\[
 \Gr(e,d)_\Fun \quad = \quad \Proj \ \bigl(\,\bpquot{\Fun[x_0,\dotsc,x_m]}{\cR} \,\bigr)
\]
is an $\Fun$-model of $\Gr(e,d)$ that comes together with a closed embedding $\Gr(e,d)_\Fun\hookrightarrow\SP^m_\Fun$. See Figure 1 of \cite{L12a} for an illustration of $\Gr(2,4)_\Fun$.

Let $Q$ be a quiver, $M$ be a representation of $Q$ with basis $\cB=\bigcup\cB_i$ and $\ue$ a dimension vector for $Q$. Though the following works also for other \emph{blueprints of definition}, we will restrict our explanation for simplicity to the case of models over $\Fun$. Namely, we assume that the $\C$-linear maps $M_\alpha$ correspond to matrices with integral coefficients w.r.t.\ the basis $\cB$. We call such a basis an \emph{integral basis of $M$}\index{Quiver representation!integral basis}. An integral basis yields an integral model $\Gr_\ue(M)$ of the complex quiver Grassmannian $\Gr_\ue(M)_\C$.

Define $\norm\ue=\sum e_i$, $\norm\ud=\sum d_i$ and $m=\binom{\norm\ud}{\norm\ue}-1$. Consider the series of closed embeddings
\[
 \Gr_\ue(M) \quad \longrightarrow \quad \prod_{i\in Q_0} \ \Gr(e_i,d_i) \quad \longrightarrow \quad \Gr(\norm\ue,\norm\ud) \quad \longrightarrow \quad \SP^m_\Z.
\]
As explained in Section \ref{subsection: closed subschemes}, this defines the $\Fun$-model $\Gr_\ue(M)_\Fun$ of $\Gr_\ue(M)$. A paper that explains the following facts is in preparation. See Section 2.1 in \cite{blueprints2} for the definition of the rank for an arbitrary blue scheme.

\begin{lemma}
 Every connected component of $\Gr_\ue(M)_\Fun$ is of rank $0$. 
\end{lemma}

\begin{prop}
 The set $\Gr^*_\ue(M,\Fun)$ of naive $\Fun$-rational points of $\Gr_\ue(M)$ is naturally a subset of the Weyl extension $\cW(\Gr_\ue(M)_\Fun)$.
\end{prop}

The following result extends to a larger class of quiver Grassmannians, but since at the point of writing it is not clear to me how to overcome certain technical assumptions, I restrict to quote the following special case that is independent of these assumptions. 

\begin{thm}\label{thm: Euler characteristic and F1-rational points}
 Let $Q$ be a tree and $M$ a representation of $Q$ with an integral basis $\cB=\bigcup\cB_i$ such that all linear maps $M_\alpha$ correspond to invertible diagonal matrices. Then the Euler characteristic of $\Gr_\ue(M,\C)$ equals the cardinality of $\cW(\Gr_\ue(M)_\Fun)$. 
\end{thm}

The technique of proof is to establish Schubert decompositions into affine spaces, which are already defined for the $\Fun$-model of the quiver Grassmannian (cf.\ the author's paper \cite{L12a}). Each Schubert cell contributes the value $1$ to the Euler characteristic and one point to the Weyl extension $\cW(\Gr_\ue(M)_\Fun)$. Note that Haupt's methods to calculate Euler characteristics of quiver Grassmannians apply to the class considered in the theorem, see \cite{Haupt12}.

We derive the following consequence for the set $\Gr^*_\ue(M,\Fun)$.

\begin{cor}\label{thm: Euler characteristic and naive f1-rational points}
 If in the context of Theorem \ref{thm: Euler characteristic and F1-rational points}, every linear map $M_\alpha$ corresponds to the identity matrix, then $\Gr^*_\ue(M,\Fun)=\cW(\Gr_\ue(M)_\Fun)$. Consequently, the Euler characteristic of $\Gr_\ue(M,\C)$ equals the number of elements of $\Gr^*_\ue(M,\Fun)$.
\end{cor}


\section{Arithmetic curves}
\label{section: the arithmetic line}

In this section, we will define a locally blueprinted space $X=\overok$ that captures properties that are expected from the Arakelov compactification\index{Arakelov theory} of $\Spec O_K$ where $O_K$ is the ring of algebraic integers of a number field $K$. The space $X$ behaves like a curve over $\Funn$ where $n$ is the number of roots of unity in $K$. Most notably, its self-product $X\times_\Funn X$ is $2$-dimensional. We will explain this in the following. More details on the special case $K=\Q$ and $X=\overz$ can be found in \cite{blueprints-mpi}.

\subsection{The naive definition of $\overok$}
\label{subsection: naive definition of spec ok}

There is a clear expectation of what an \emph{arithmetic curve $\overok$}\index{Arithmetic curve} should be. We describe this in the following. This viewpoint is naive in the sense that it is not clear, which category is suitable to give a rigorous definition for $X=\overok$. Namely, in analogy with complete smooth curves over a field, we expect that the underlying topological space of $X$ consists of a unique generic point $\eta$ (corresponding to the discrete norm $\norm \ _0$) and a closed point $p$ for every (non-discrete) place $\norm\ _p$ of the ``function field'' $K$ of $X$. The closed sets of $X$ are finite sets $\{p_1,\dotsc,p_n\}$ of non-trivial places and $X$ itself. Further, there should be a structure sheaf $\cO_X$, which associates to the open set $U=X-\{p_1,\dotsc,p_n\}$ the set
\[
 \cO_X \bigl( U \bigr) \quad = \quad \left\{ \ a \in K \ \left| \ \norm{a}_q\leq 1\text{ for all }q\notin\{p_1,\dotsc,p_n\}\ \right.\right\}
\]
of regular functions. As a consequence, the global sections are $\Gamma(X,\cO_X)=\cO_X(X)=\{0\}\cup\mu_n$ where $\mu_n$ is the group of roots of unity of $K$, which should be thought of as the constants of $X=\overok$. The stalks of $\cO_X$ are 
\[
 \cO_{X,p} \quad = \quad \left\{ \ a \in K \ \left| \ \norm{a}_p\leq 1\ \right.\right\},
\]
with ``maximal ideals''
\[
 \fm_p \quad = \quad \left\{ \ a \in K \ \left| \ \norm{a}_p< 1\ \right.\right\}
\]
for every place $p$. One observes that $X$ is indeed an extension of the scheme $\Spec O_K$, i.e.\ the restriction of $X$ to the open subset of non-archimedean places can be identified with $\Spec O_K$. The problem with this definition is that the set $\cO_X(U)$ is not a sub\emph{ring} of $K$ if $U$ contains an archimedean place, and neither is the stalk $\cO_{X,p}=\{a\in K|\norm a_p\leq1\}$ at an archimedean place $p$.

Therefore, it is not clear as what kind of structure the sets $\cO_X(U)$ should be considered. Note that all of these sets are monoids with zero. This emphasizes the viewpoint that $\overok$ should be an object defined in terms of $\Fun$-geometry, whose basic idea is to forget or, at least, to loosen addition. There are different suggestions by Durov (cf.\ \cite{Durov07}), Haran (cf.\ \cite{Haran07, Haran09}) and Takagi (cf.\ \cite{Takagi12b}) for $\overz$, from which the most promising seems to be the one in \cite{Haran09} since the self-product of $\overz$ over $\Fun$ yields an interesting and complicated space. In the following section, we will present a model of $\overz$ and more generally of $\overok$ as a locally blueprinted space, whose self-product behaves in analogy to a curve over a field.

\subsection{Arithmetic curves as locally blueprinted spaces}
\label{subsection: spec z as a locally blueprinted space}

The crucial observation is that the (prime) ideals of a discrete valuation ring $R$ are the same as the (prime) ideals of the underlying multiplicative monoid $R^\bullet$. Indeed, if $p$ is a uniformizer of $R$, then $R^\bullet$ is isomorphic to the free monoid $A[p]$ over $A=\{0\}\cup R^\times$. Since $A$ is a blue field, the ideals of $A[p]$ are $(0)$ and $(p^i)$ for $i\in\N$. The only prime ideals are $(0)$ and $\fm=(p)$. Since the inclusion $R^\bullet\to R$ of blueprints induces a bijection between the (prime) ideals, this result is still true for any blueprint $B$ in between $R^\bullet$ and $R$, i.e.\ for any pre-addition for $R^\bullet$ that consists of additive relations that hold in $R$. 

This observation leads to the following definition of the compactification of the spectrum of a ring $O_K$ of algebraic integers of a number field $K$. We will define the \emph{completion of $\Spec O_K$}\index{Completion of $\Spec O_K$} as a locally blueprinted space $\overok=(X,\cO_X)$ where $X$ is the topological space from the previous section.

Recall further from Section \ref{subsection: naive definition of spec ok} that the monoid $\Gamma=\{a\in K|\norm a_p\leq 1\text{ for all }p\in X\}$ equals $\{0\}\cup\mu_n$ where $\mu_n$ is the group of roots of unity of $K$. If we endow $\Gamma$ with the pre-addition $\cR$ that contains all relations that hold between the elements of $\Gamma$ in $O_K$, then $\bpquot\Gamma\cR$ is isomorphic to $\Funn$. Thus we will think of $\overok$ as a curve defined over $\Funn$, and require the same additive relations $\cR$ for the structure sheaf $\cO_X$.

Namely, we define for $U=X-\{p_1,\dotsc,p_n\}$,
\[
 A(U) \quad = \quad \Bigl\{ \ a \in K \ \Bigl| \ \norm{a}_q\leq 1\text{ for all }q\notin\{p_1,\dotsc,p_n\}\ \Bigr\}
\]
and the structure sheaf $\cO_X$ by $\cO_X(U)=\bpgenquot{A(U)}{\cR}$ together with the natural inclusions $\res_{U|V}:\cO_X(U)\hookrightarrow\cO_X(V)$ as restriction maps. It turns out that $\cO_X$ is a sheaf and that indeed $\Gamma(X,\cO_X)=\bpquot\Gamma\cR\simeq\Funn$. The stalk at a place $p$ is $\cO_{X,p}=\bpgenquot{A_p}{\cR}$ where
\[
 A_p \quad = \quad \Bigl\{ \ a \in K \ \Bigl| \ \norm{a}_p\leq 1\ \Bigr\}.
\]
The blueprint $\cO_{X,p}$ is local with the maximal ideal $\fm_p=\{ \, a \in K \, | \, \norm{a}_p< 1\, \}$. The only other prime ideal of $\cO_{X,p}$ is $0$. For a non-archimedean place $p$, this follows from the considerations at the beginning of this section, since $O_{K,p}$ is a discrete valuation ring. For an archimedean place $p$, the condition $IB\subset I$ implies that the ideals of $\cO_{X,p}$ are of the form
\[
 I^c_r \quad = \quad \bigl\{ \ a\in K \ \bigl| \ \norm a_p \leq r \ \bigr\}  \qquad \text{or} \qquad I^o_r \quad = \quad \bigl\{ \ a\in K \ \bigl| \ \norm a_p < r \ \bigr\}
\]
for some $r\in[0,1]$ (for $I^c_r$) resp.\ $r\in (0,1]$ (for $I^o_r$). It is clear that the only prime ideals among these are $(0)=I^c_0$ and $\fm_p=I^o_1$. Note that $I^c_r=I^o_r$ if $r$ is not the norm $\norm a_p$ of an element $a\in K$.

In particular, we see that $\overok=(X,\cO_X)$ is a locally blueprinted space that is defined over $\Funn$ and whose dimension is $1$ (as a topological space). Its structure sheaf is coherent in the sense that $X$ has an open covering $\{U_i\}$ such that for every $V\subset U_i$, the restriction map $\cO_X(U_i)\to\cO_X(V)$ is a localization while its global sections are equal to the field of definition. Concerning these aspects, $\overok$ behaves like a complete curve over the cyclotomic field extension $\Funn$ of $\Fun$.

\subsection{The arithmetic surface $\overok\times_\Funn\overok$}
\label{subsection: the arithmetic surface}

 Since for every place $p$, the residue field $\kappa(p)$ can be embedded into fields of any characteristic, Theorem \ref{thm: fibre products} implies the following.

 \begin{prop}
  The ``arithmetic surface'' $\overok\times_\Funn \overok$ is a topological space of dimension $2$. \qed
 \end{prop}

\begin{rem}
 The locally blueprinted space $\overok$ comes together with a morphism $\varphi:\Spec O_K\to\overok$, which is a topological embedding. The image $Y$ of $\varphi$ (as a locally blueprinted space) is thus homeomorphic to $\Spec O_K$ and the underlying monoids of the structure sheaf are the same, only the pre-additions of $\Spec O_K$ and $Y$ differ. 
\end{rem}


\section{$K$-theory}
\label{section: k-theory}

As explained in Part \ref{partI} of this text, the $K$-theory\index{K-theory@$K$-theory} of $\Fun$ is expected to equal the stable homotopy of the sphere spectrum, i.e.\ $K(\Fun)=\pi^\st(\S^0)$. Deitmar defines $K$-theory for monoids and makes sense of this formula in \cite{Deitmar07}. Chu, Santhanam and the author extend in \cite{CLS12} the definition of $K$-theory to monoidal schemes, which requires a close inspection of categories of $\cO_X$-modules of monoidal schemes. As a particular result, it turns out that there exists indeed a $K$-theory spectrum $\cK(X)$ for every monoidal scheme $X$, which comes with a ring structure. This implies that the $K$-theory of $\Spec\Fun$ equals the stable homotopy of the sphere spectrum as a ring. 

The definition of $K$-theory in \cite{CLS12} generalizes to blue schemes. We will summarize this theory and its extension to blue schemes in the following.

\subsection{Blue modules}
\label{subsection: blue modules}

Let $M$ be a pointed set. We denote the base point of $M$ by $\ast$. A \emph{pre-addition on $M$}\index{Pre-addition} is an equivalence relation $\cP$ on the semigroup $\N[M]=\{\sum a_i|a_i\in M\}$ of finite formal sums in $M$ with the following properties (as usual, we write $\sum m_i\=\sum n_j$ if $\sum m_i$ stays in relation to $\sum n_j$):
\begin{enumerate}
 \item $\sum m_i\=\sum n_j$ and $\sum p_k\=\sum q_l$ implies $\sum m_i+\sum p_k\=\sum n_j+\sum q_l$,
 \item $\ast\=(\text{empty sum})$, and 
 \item if $m\=n$, then $m=n$ (in $M$).
\end{enumerate}
Let $B=\bpquot A\cR$ be a blueprint. A \emph{blue $B$-module}\index{Blue $B$-module} is a set $M$ together with a pre-addition $\cP$ and a \emph{$B$-action $B\times M\to M$}\index{B-action@$B$-action}, which is a map $(b,m)\mapsto b.m$ that satisfies the following properties:
\begin{enumerate}
 \item $1.m=m$, $0.m=\ast$ and $a.\ast=\ast$,
 \item $(ab).m=a.(b.m)$, and
 \item $\sum a_i\=\sum b_j$ and $\sum m_k\=\sum n_l$ implies $\sum a_i.m_k\=\sum b_j.n_l$.
\end{enumerate}
A \emph{morphism of blue $B$-modules $M$ and $N$}\index{Morphism!of blue $B$-modules} is a map $f:M\to N$ such that
\begin{enumerate}
 \item $f(a.m)=a.f(m)$ for all $a\in B$ and $m\in M$ and 
 \item whenever $\sum m_i\=\sum n_j$ in $M$, then $\sum f(m_i) \= \sum f(n_j)$ in $N$. 
\end{enumerate}
This implies in particular that $f(\ast)=\ast$. We denote the category of blue $B$-modules by $\Mod B$.

\begin{prop}
 The category $\Mod B$ is complete and cocomplete. The trivial blue module $0=\{\ast\}$ is an initial and terminal object of $\Mod B$.
\end{prop}

The existence of the zero object in $\Mod B$ implies that for any two blue $B$-modules $M$ and $N$, there is a unique morphism $0:M\to N$ that factorizes through $0=\{\ast\}$. This allows us to define the (co)kernel of a morphism as the difference (co)kernel with the $0$-morphism. In particular, we obtain a notion of (short) exact sequences.

Note that if $B$ is a ring, then a $B$-module is naturally a blue $B$-module, but not vice versa since a blue $B$-module does not have to be closed under addition. The product of $B$-modules and blue $B$-modules coincide, namely, $\prod M_i$ is the Cartesian product with the induced $B$-module structure. 

This is not true for coproducts. Let $B$ be a blueprint and $(M_i)$ a family of blue $B$-modules. Then the coproduct over the $M_i$ is the wedge product $\bigvee M_i$ that identifies the base points $\ast\in M_i$. A $B$-module $M$ is \emph{free of rank $\kappa$}\index{Blue $B$-module!free} if there is an index set of cardinality $\kappa$ such that $M$ is isomorphic to $\bigvee_{i\in I} B$. A $B$-module $P$ is \emph{projective}\index{Blue $B$-module!projective} if the functor $\Hom(P,-)$ is right exact. 

As in the case of usual modules over a ring, every free $B$-module is projective. If $k$ is a blue field, then every projective $k$-module is free.

\subsection{Blue sheaves}
\label{subsection: blue sheaves}

Let $X$ be a blue scheme. A \emph{blue $\cO_X$-module}\index{Blue $\cO_X$-module} is a sheaf $\cM$ on $X$ that associates to every open subset $U$ of $X$ a blue $\cO_X(U)$-module $\cM(U)$ such that the $\cO_X(U)$-actions are compatible with the restriction maps. A \emph{morphism $\varphi:\cM\to \cN$ of $\cO_X$-modules}\index{Morphism!of blue $\cO_X$-modules} is a family of $\cO_X(U)$-module morphisms $\cM(U)\to\cN(U)$ for open subsets $U$ of $X$ that commute with the restriction maps of $\cM$ and $\cN$.

The category of $\cO_X$-modules inherits the properties of blue $B$-modules: it is complete and cocomplete and it has a zero object $0$. 

A blue $\cO_X$-module $\cM$ on $X$ is \emph{locally free of rank $\kappa$}\index{Blue $\cO_X$-module!locally free} if $\cM(U)$ is a free $\cO_X(U)$-module of rank $\kappa$ for all affine opens $U$ of $X$. A blue $\cO_X$-module $\cM$ on $X$ is \emph{locally projective}\index{Blue $\cO_X$-module!locally projective} if there exists an affine open covering $\{U_i\}$ of $X$ such that $\cM(U_i)$ is a projective $\cO_X(U_i)$-module for all $i$. A locally projective blue $\cO_X$-module $\cP$ is \emph{finitely generated}\index{Blue $\cO_X$-module!finitely generated} if there exists an epimorphism $\cM\to\cP$ from a locally free blue $\cO_X$-module $\cM$ of finite rank. We denote the category of finitely generated locally projective blue $\cO_X$-modules together with all $\cO_X$-module morphisms by $\Bun X$.

\begin{ex}
 In contrast to the theory of Grothendieck schemes, locally projective $\cO_X$-modules are in general not locally free. This can be seen in the following example. Let $B=\{0,e,1\}$ be the monoid with one idempotent $e^2=e$. Then the ideal $\fp=\{0,e\}$ of $B$ is a projective $B$-module. The spectrum of $B$ consists of the two prime ideals $(0)$ and $\fp$ and the open sets $\emptyset$, $\{(0)\}$ and $X$. We define a sheaf $\cP$ by $\cP(\{(0)\})=\cP(X)=\fp$ together with the obvious restriction maps. Then $\cP$ is a locally projective $\cO_X$-module.
 
 An open covering $\{U_i\}$ of $X$ necessarily contains $X$. Since $\cP(X)=\fp$ is not a free $B$-module, the locally projective $\cO_X$-module $\cP$ is not locally free.
\end{ex}

\subsection{Blue $K$-theory}
\label{subsection: blue k-theory}

Let $X$ be a blue scheme. Recall that a monomorphism is called \emph{normal}\index{Normal monomorphism} if it is the kernel of its cokernel and that an epimorphism is called \emph{normal}\index{Normal epimorphism} if it is the cokernel of its kernel. 

The important properties for the application of Quillen's $Q$-construction are the following. 

\begin{lemma}
 Normal monomorphisms are closed under composition and under base change along normal epimorphisms. Normal epimorphisms are closed under composition and under base change along normal monomorphisms. 
\end{lemma}

With this we can define the category $Q\Bun X$: its objects are finitely generated locally projective sheaves and its morphisms are isomorphism classes of diagrams 
\[
 \cM \stackrel f\longleftarrow \cP \stackrel g\longrightarrow \cN
\]
 where $f$ is a normal epimorphism and $g$ is a normal monomorphism. Two diagrams 
\[
 \cM\stackrel f\longleftarrow\cP \stackrel g\longrightarrow \cN \qquad \text{and} \qquad \cM\stackrel {f'}\longleftarrow\cP' \stackrel {g'}\longrightarrow \cN
\]
 are isomorphic if there exists an isomorphism $h:\cP\to \cP'$ such that $f=f'\circ h$ and $g=g'\circ h$.

The \emph{$i$-th blue $K$-group of $X$}\index{Blue $K$-theory} is defined as the homotopy group
\[
 K_i^\bl(X) \quad = \quad \pi_{i+1}|Q\Bun X|
\]
of the geometric realization of the nerve of $Q\Bun X$. We subsume some results from Sections 5.3 and 5.4 of \cite{CLS12} in the following theorem.

\begin{thm}
 Let $X$ be a monoidal scheme. Then there is a natural symmetric ring spectrum $\cK^\bl(X)$ whose stable homotopy groups are the blue $K$-groups of $X$. The induced multiplication on $K^\bl(X)=\bigoplus_{i\geq0} K_i^\bl(X)$ is compatible with the tensor product of locally projective $\cO_X$-modules.
\end{thm}

Finally, we obtain the expected relation with the stable homotopy of the sphere spectrum\index{Stable homotopy of spheres}.

\begin{thm}[{\cite[Thm.\ 5.9]{CLS12}}]
 The $K$-theory spectrum $\cK^\bl(\Fun)$ is isomorphic to the sphere spectrum $\S$ as a symmetric ring spectrum. In particular, $K_i^\bl(\Fun)=\pi^\st_i(\S)$ is the $i$-th stable homotopy group of the sphere spectrum.
\end{thm}

\begin{ex}
We conclude this exposition with an example that compares the Grothendieck groups of different $K$-theories. The base extension $X^+_\Z\to X$ of a blue scheme $X$ to $\Rings$ and the inclusion of the category of $\cO_X$-modules into the category of blue $\cO_X$-modules induces ring homomorphisms
\[
 K_0^\bl(X) \quad\longrightarrow\quad K_0^\bl(X^+_\Z) \quad\longrightarrow\quad K_0(X^+_\Z).
\]
We inspect them in case in case of $X=\Spec\Funn$. Then $X^+_\Z$ is the Dedekind scheme $\Spec \Z[\zeta_n]$ where $\Z[\zeta_n]$ is the ring of algebraic integers of the cyclotomic field extension $\Q[\zeta_n]$ of $\Q$.

Since $\Funn$ is a blue field, all locally projective $\cO_X$-modules are free, and thus $K_0^\bl(X)=\Z$. The locally projective blue $\cO_{X^+_\Z}$-modules of rank $1$ coincide with the invertible sheaves on $X^+_\Z$ in the usual sense, while locally projective blue $\cO_{X^+_\Z}$-modules of higher rank decompose into coproducts of invertible sheaves. Therefore $K_0^\bl(X^+_\Z)=\Cl \Q[\zeta_n]\times \Z$ since is the ideal class group $\Cl \Q[\zeta_n]$ of $\Q[\zeta_n]$ is isomorphic to the Picard group of $X^+_\Z$. Thus the above sequence $K_0^\bl(X)\to  K_0^\bl(X^+_\Z)\to K_0(X^+_\Z)$ becomes in this case
\[
 \Z  \quad\longrightarrow\quad \Cl \Q[\zeta_n] \times\Z \quad\stackrel\id\longrightarrow\quad \Cl \Q[\zeta_n] \times\Z.
\]
\end{ex}

\begin{rem}
 There are interesting number theoretic application for a $K$-theory of arithmetic curves\index{Arithmetic curve} $\overok$. One could think about an interpretation of Gauss reciprocity as Weil reciprocity for $\overz$ or about applications in Iwasawa theory (see, for instance, Greither and Popescu's proof of the equivariant main conjecture in \cite{Greither-Popescu11}). 

 It is not clear yet how to define such a $K$-theory. It seems that blue $K$-theory as defined in this section applied to the arithmetic curves $\overok$ from Section \ref{section: the arithmetic line} is not yet the right notion. However, there is hope that a variation of the definition given in this paper will be suited for applications to number theory.
\end{rem}


\section{Connection to other geometries}
\label{section: connection to other geometries}

In the previous sections, we have already seen that monoidal schemes (cf.\ Section \ref{subsubsection: Relation to usual schemes and monoidal schemes}) are naturally blue schemes and that $\B_1$-algebras and sesquiads are a particular sort of blueprints (cf.\ Section \ref{subsubsection: Idempotent semirings and sesquiads}). There are connections to other geometric theories. 

By name, a \emph{Frobenius scheme} in $\Sch_\Fun$ defines a $\Lambda$-schemes after Borger (see Section \ref{subsection: lambda-schemes}). Blue schemes define scheme relative to the category of blue $\Fun$-modules in the sense of To\"en and Vaqui\'e (see Section \ref{subsection: relative schemes after toen and vaquie}). A \emph{(general) blue scheme} defines a log scheme and there are easily formulated conditions on a general blue scheme that imply that the associated log scheme is a fine log scheme (see Section \ref{subsection: log schemes}). Finally, there is a need for a theory of congruence schemes for blueprints, which unites the different approaches of Berkovich, Deitmar and Lescot. See Section \ref{subsection: congruence schemes} for some remarks on such a theory.

\subsection{$\Lambda$-schemes}
\label{subsection: lambda-schemes}

$\Lambda$-algebraic geometry, as developed by Jim Borger in \cite{Borger11a} and \cite{Borger11b}, has its own viewpoint on $\Fun$-geometry, which is explained in \cite{Borger09}. The idea is to define an $\Fun$-scheme as a scheme $X$ together with a $\Lambda$-structure, which should be thought of as the descend datum to $\Fun$. The base extension of the \emph{$\Lambda$-scheme $X$} from $\Fun$ to $\Z$ is the \emph{scheme $X$}, considered without the $\Lambda$-structure. 

So far, $\Lambda$-structures have been defined for flat schemes only. Therefore, we make the general assumption that all schemes and rings are flat throughout Section \ref{subsection: lambda-schemes}. A \emph{$\Lambda$-structure for a flat scheme $X$}\index{Lambda-structure@$\Lambda$-structure} is a system of pairwise commuting \emph{Frobenius lifts $\psi_p:X\to X$} for every prime $p$, i.e.\ the $\psi_p$ are endomorphisms of $X$ such that the diagrams
\[
 \xymatrix@C=4pc{X\otimes_\Z\F_p \ar[r]^{\Frob_p}\ar[d] & X\otimes_\Z\F_p \ar[d] \\ X \ar[r]_{\psi_p} & X}
\]
commute where $\Frob_p$ is the Frobenius endomorphism of $X\otimes_\Z\F_p$ and the vertical maps are the natural inclusions as fibres at $p$. We will call a flat scheme with a $\Lambda$-structure a \emph{$\Lambda$-scheme}\index{Lambda-scheme@$\Lambda$-scheme}. A morphism of $\Lambda$-schemes is a morphism of schemes that commutes with the endomorphisms $\psi_p$. This defines the category $\Lambda-\Sch^+_\Z$ of (flat) $\Lambda$-schemes.

\begin{ex}
 Every toric variety\index{Toric variety} can be considered as a $\Lambda$-scheme. Let $\Delta$ be a fan. Then we write $A_\tau$ for the multiplicatively written dual cone of a cone $\tau$ of $\Delta$. Let $\Z[A_\tau]$ be the semi-group ring generated by $A_\tau$. Then the association $a\mapsto a^p$ for $a\in A_\tau$ defines a ring endomorphism $F_p:\Z[A_\tau]\to\Z[A_\tau]$ for every prime $p$ and every cone $\tau$. This glues to endomorphisms $\psi_p$ of the toric variety $X(\Delta)=\colim \Spec\Z[A_\tau]$ for every prime $p$. This turns $X(\Delta)$ into a $\Lambda$-scheme.
\end{ex}

 A more general construction of $\Lambda$-schemes is obtained by gluing $\Lambda$-rings along $\Lambda$-morphisms that define open immersions of affine schemes. Namely, a \emph{$\Lambda$-structure for a flat ring}\index{Lambda-structure@$\Lambda$-structure} is a system of pairwise commuting endomorphisms $F_p:R\to R$ whose reduction modulo $p$ is the Frobenius endomorphism $a\mapsto a^p$ on $R/(p)$. A \emph{flat $\Lambda$-ring}\index{Lambda-ring@$\Lambda$-ring} is a flat ring with a $\Lambda$-structure. A homomorphism of flat $\Lambda$-rings is a ring homomorphism that commutes with the endomorphisms $F_p$. The functor $\Spec$ associates to a flat $\Lambda$-ring naturally an affine $\Lambda$-scheme and to a homomorphism of flat $\Lambda$-rings a homomorphism of $\Lambda$-schemes. This allows us to glue spectra of flat $\Lambda$-rings along open immersions that are $\Lambda$-morphisms, which yields a $\Lambda$-scheme.

\begin{ex}[A $\Lambda$-scheme without an affine cover]\label{ex: lambda-scheme without affine cover}
 Note that there are $\Lambda$-schemes that cannot be covered by spectra of flat $\Lambda$-rings. The following example is taken from \cite[para.\ 2.6]{Borger09}. The toric $\Lambda$-structure of the projective line $\SP^1_\Z$ descends to the quotient $X$ of $\SP^1_\Z$ that identifies the torus fixed points $0$ and $\infty$. This defines a $\Lambda$-structure on $X$, but there is no open immersion $\varphi:U\to X$ from an affine $\Lambda$-scheme $U$ such that $\varphi$ is a $\Lambda$-morphism and such that the point $0=\infty$ is in the image of $\varphi$.
\end{ex}

\subsubsection{Frobenius blueprints}
\label{subsubsection: frobenius blueprints}

We will explain the connection of $\Lambda$-rings to blueprints. To keep the picture simple, we assume that all blueprints $B$ are \emph{flat}\index{Blueprint!flat}, i.e.\ the associated ring $B^+_\Z$ is flat. A \emph{Frobenius blueprint}\index{Blueprint!Frobenius blueprint} is a reduced blueprint $B$ with the property that for every prime $p$, a relation $\sum a_i\=\sum b_j$ implies that $\sum a_i^p\=\sum b_j^p$ holds in $B$. Note that this condition is testable on a set $S$ of generators for $\cR=\gen S$. In other words, the multiplicative map $\Frob_p:B\to B$ that sends $a$ to $a^p$ is a blueprint morphism for every prime $p$. We call $\Frob_p$ the \emph{$p$-th power Frobenius of $B$}. We denote the full subcategory of $\bp$ whose objects are Frobenius blueprints by $\Frob-\bp$.

\begin{ex}
 Every monoid with zero is trivially a Frobenious blueprint. 

 A more interesting example is the following. Let $\mu_{n,0}$ be the monoid of the $n$-th roots of unity with zero, i.e.\ $\mu_n=\mu_{n,0}-\{0\}$ is a cyclic group of order $n$. Let $\cR$ be the pre-addition that is generated by the relations
 \[
  \sum_{a\in H} \ a \quad \= \quad \underbrace{1+\dotsb+1}_{\#H-\text{times}}
 \]
 for every subgroup $H$ of $\mu_{n,0}$. Then $B=\bpquot{\mu_{n,0}}{\cR}$ is a Frobenius blueprint. Indeed, since taking the $p$-th power defines a group homomorphism from $\mu_n$ onto some subgroup, the relation $\sum_{a\in H} a^p\=1^p+\dotsb+ 1^p$ is equivalent to a multiple of $\sum_{a\in H'} a\=1+\dotsb+ 1$ for some subgroup $H'$ of $H$.

 Note that this Frobenius blueprint appears as a subblueprint of the big Witt vectors of a field $k$ containing all $n$-th roots of unity. 
\end{ex}

A \emph{Frobenius scheme}\index{Blue scheme!Frobenius scheme} is a blue scheme $X$ such that for every open $U$ of $X$, $\cO_X(U)$ is a Frobenius blueprint. Therefore, a Frobenious scheme $X$ comes together with a family of pairwise commuting endomorphisms $\Frob_p:X\to X$ where $p$ ranges through all prime numbers. We denote the full subcategory of $\Sch_\Fun$ whose objects are Frobenius schemes by $\Frob-\Sch_\Fun$.

\subsubsection{Comparing $\Lambda$-rings and Frobenius blueprints}
\label{subsubsection: comparing lambda-rings and frobenius blueprints}

Let $B$ be a Frobenius blueprint. Then the $p$-th power Frobenius extends uniquely to a ring endomorphism $F_p:B_\Z^+\to B_\Z^+$ for every prime $p$. This defines a $\Lambda$-structure on $B_\Z^+$, which is functorial in $B$. Since $B$ is flat, $B_\Z^+$ is a flat and hence a flat $\Lambda$-ring. This yields a functor
\[
 \Lambda: \quad \Frob-\bp \quad \longrightarrow \quad \Lambda-\Rings.
\]
from the full subcategory of Frobenius blueprints in $\bp$ to the category of $\Lambda$-rings. This functor extends to a functor
\[
 \Lambda: \quad \Frob-\Sch_\Fun \quad \longrightarrow \quad \Lambda-\Sch_\Z^+.
\]
from Frobenius schemes to $\Lambda$-schemes. The image of $\Lambda$ is contained in the full subcategory $\cL$ of $\Lambda$-schemes that have an affine open covering by spectra of $\Lambda$-rings.

Given a flat $\Lambda$-ring $R$ with endomorphisms $F_p$, we define the subset
\[
 A \quad = \quad \{ \ a\in R \ | \ F_p(a)= a^p \text{ for all primes }p\ \},
\]
which contains $0$ and $1$ and is multiplicatively closed. Let $\cR$ be the pre-addition on $A$ that contains all relations $\sum a_i\=\sum b_j$ that hold in $R$. Then $B=\bpquot A\cR$ is a blueprint. Since $R$ is flat, $B$ is flat, and by the very definition of $A$, $B$ is a Frobenius blueprint. This association is functorial in $R$, which yields a functor
\[
 \Frob: \quad \Lambda-\Rings \quad \longrightarrow \quad \Frob-\bp,
\]
which extends to a functor
\[
 \Frob: \quad \cL \quad \longrightarrow \quad \Frob-\BSch.
\]
However, since in general a $\Lambda$-scheme cannot be covered by spectra of $\Lambda$-rings (cf.\ Example \ref{ex: lambda-scheme without affine cover}), this functor fails to extend to the whole category $\Lambda-\Sch_\Z^+$.

Clearly, every Frobenius scheme that comes from a $\Lambda$-scheme is cancellative. Note that not every Frobenius scheme is cancellative, see Example \ref{ex: non-cancellative Frobenius scheme} below. We denote the full subcategory of cancellative Frobenius schemes by $\Frob-\Sch_\Fun^\canc$. The following fact is easily verified.

\begin{prop}
 The composition $\Lambda\circ\Frob$ is isomorphic to the identity functor on $\Frob-\Sch_\Fun^\canc$.
\end{prop}

However, there are $\Lambda$-schemes in $\cL$ that do not come from Frobenius schemes, see Example \ref{ex: the chebychev line}. This can be given the following interpretation: $\Fun$-geometry (with Frobenius actions) is a part of $\Lambda$-algebraic geometry, but $\Lambda$-algebraic geometry is richer than $\Fun$-geometry.

\begin{ex}[Non-cancellative Frobenius blueprints]\label{ex: non-cancellative Frobenius scheme}
 Let $\B_1=\bpgenquot{\{0,1\}}{1+1\=1}$ be the idempotent semi-ring with two elements. Then $\B_1$ is obviously a Frobenius blueprint. Note that $(\B_1)^+_\Z$ is the zero ring, which means that $\Lambda\circ\Frob(\B_1)$ is the zero blueprint. More generally, every blueprint of the form $B=\bpgenquot{A}{1+1\=1}$ is a Frobenius blueprint whose associated $\Lambda$-ring is the zero ring.
\end{ex}

\begin{ex}[The Chebychev line]\label{ex: the chebychev line}
 Consider the action $t\mapsto t^{-1}$ on the ring $\Z[t^{\pm 1}]$. The invariant subring is $\Z[x]$ with $x=t+t^{-1}$. We define for every prime $p$ the endomorphism $F_p$ by $F_p(x)=t^p+t^{-p}$, which is a polynomial in $x$. This endows the affine line $\SA^1_\Z=\Spec\Z[x]$ with a $\Lambda$-structure, and this $\Lambda$-scheme $X$ is called the \emph{Chebychev line}\index{Chebychev line}, cf.\ Section 2.5 in \cite{Borger09}.

 The only elements $a$ of $\Z[x]$ that satisfy $F_p(a)=a^p$ are $0$ and $1$. Therefore $\Frob (X)=\A^0_\Fun$. Since $\Lambda(\A^0_\Fun)=\SA^0_\Z$ (together with its unique $\Lambda$-structure), this shows that not every $\Lambda$-scheme with a covering by affine $\Lambda$-schemes comes from a Frobenius scheme.
\end{ex}

\subsection{Relative schemes after To\"en-Vaqui\'e}
\label{subsection: relative schemes after toen and vaquie}

 To\"en and Vaqui\'e develop in \cite{Toen-Vaquie09} the notion of a scheme relative to any complete and cocomplete closed symmetric monoidal category $\cC$. If $\cC$ is understood, we say briefly \emph{relative scheme} for a scheme relative to $\cC$. An \emph{affine relative scheme} is an object of the dual category $\Aff(\cC)$ of the category $\Comm(\cC)$ of commutative monoids in $\cC$. Let $\spec:\Comm(\cC)\to\Aff(\cC)$ be the contravariant isomorphism. A morphism is an open immersion if it is dual to an epimorphism $f:B\to C$, such that $\blanc\otimes_BC$ commutes with finite limits and colimits ($f$ is \emph{flat}\index{Morphism!flat}) and the canonical map
 \[
 \Psi_\cD: \quad \colim \ \Hom_B (C,\cD) \quad \longrightarrow \quad \Hom_B(C,\colim \cD)
 \]
 is a bijection for every directed system $\cD$ in $\cC$ ($f$ is \emph{of finite presentation}\index{Morphism!of finite representation}).

 A family $\{\spec B_i\to \spec B\}_{i\in I}$ of open immersions is a \emph{covering of $\spec B$} if there is a finite subset $J\subset I$ such that the functor
\[
 \prod_{j\in J} \ \blanc\otimes_BB_j \ : \quad \Mod B \quad \longrightarrow \quad \prod_{j\in J} \ \Mod B_j
\]
is conservative (i.e.\ $f:M\to N$ is an isomorphism if $\Phi(f)$ is an isomorphism). This endows $\Aff(\cC)$ and with the structure of a site. A \emph{scheme relative to $\cC$}\index{Relative scheme} is a sheaf on the site $\Aff(\cC)$ that can be covered by open affine relative subschemes. We denote the category of schemes relative to $\cC$ by $\Sch(\cC)$.

 For instance, we can associate to a Grothendieck scheme $X$ over a ring $R$ the functor $h_X=\Hom_R(\blanc,X)$ of points, which is a sheaf on the Zariski site of affine schemes. This establishes an equivalence between the category of Grothendieck schemes over $R$ and schemes relative to the category of $R$-modules (this follows from the functorial viewpoint of schemes as developed in Demazure and Gabriel's book \cite{Demazure-Gabriel70}). Similarly, a monoidal scheme $X$ defines a sheaf $h_X=\Hom(\blanc,X)$ on the Zariski site of affine monoidal schemes (after Deitmar, \cite{Deitmar05}). This establishes an equivalence between monoidal schemes and schemes relative to the category of sets (as proven by Vezzani in  \cite{Vezzani12}).

 These equivalences find a common generalization in the theory of blueprints, though it is not anymore true for a general blue scheme $X$ that the functor $h_X=\Hom_B(\Spec\Gamma(\blanc),X)$ is a scheme relative to $\Mod B$. The reason for this is that we encounter two different Zariski sites on (the dual of) the category of blueprints: the Zariski site coming from $\Sch_\Fun$ is finer than the Zariski site coming from $\Sch(\Mod B)$. Therefore the identity functor $\bp$ defines a morphism from the latter site to the former site, which induces a functor $\cG:\Sch(\Mod B)\to\Sch_\Fun$, but not vice versa.

 However, we can define a right inverse $\cF$ on a certain subcategory of $\Sch_\Fun$, which extends the functor of points for usual schemes and monoidal schemes. Namely, we say that a morphism $f:X\to Y$ between affine blue schemes is a \emph{finite localization} if the morphism $\Gamma f:\Gamma Y\to \Gamma X$ is the localization $B\to S^{-1} B$ of $B=\Gamma Y$ at a finitely generated multiplicative subset $S$ of $B$. An affine open covering $\{U_i\to X\}$ of an affine blue scheme $X$ is \emph{conservative} if the functor $\Mod \Gamma X \quad \longrightarrow \quad \prod_{j\in I} \ \Mod \Gamma U_i$ is conservative. An affine blue scheme $X$ is \emph{with an algebraic basis} if every covering of $X$ is conservative and if the affine open subsets $U$ of $X$ that are finite localizations form a basis of the topology of $X$. 

 An \emph{algebraic presentation} of a blue scheme $X$ is a diagram $\cU$ of affine open subschemes $U_i$ of $X$ together with their inclusion maps such that all $U_i$ are with an algebraic basis, such that all inclusions $U_i\hookrightarrow U_j$ are finite localizations, such that every $U_i$ is contained in a maximal $U_j$ in $\cU$ and such that $X$ is the colimit of $\cU$. A blue scheme $X$ is \emph{algebraically presented} if it has an algebraic presentation. We denote the full subcategory of $\Sch_\Fun$ whose objects are algebraically presented blue schemes by $\Sch_\Fun^\alg$.

 Examples of algebraically presented blue schemes are usual schemes, monoidal schemes and all blue schemes of finite type over $\Fun$. 

 Note that if $\cU$ is an algebraic presentation of $X$, then applying the global sections functor $\Gamma$ and $\spec$ to each object and each morphism in $\cU$, we obtain a diagram $\spec\Gamma\cU$ of affine relative schemes. We summarize the results of Sections 9--11 in \cite{L12b}. 

 \begin{thm}
  Let $X$ be a blue scheme with algebraic atlas $\cU$. Then the colimit of the diagram $\spec\Gamma\cU$ in $\Sch(\Mod\Fun)$ exists. This defines naturally a fully faithful functor $\cF:\Sch^\alg_\Fun\to\Sch(\Mod\Fun)$. The composition $\cG\circ\cF:\Sch_\Fun^\alg\to\Sch_\Fun$ is isomorphic to the embedding of $\Sch_\Fun^\alg$ as a subcategory in $\Sch_\Fun$.
 \end{thm}

 Denote by $\Gamma\bp^\alg$ the full subcategory of $\bp$ whose objects are global blueprints $B$ such that $\Spec B$ is with an algebraic basis. Monoids and rings are examples of such blueprrints.

 \begin{cor}
  If we restrict $\cG$ to the essential image of $\cF$, then $\cF$ is a right inverse of $\cG$. As functors on $\Gamma\bp^\alg$, $\cF(X)$ and  $h_X=\Hom_B(\Spec\Gamma(\blanc),X)$ are isomorphic for every algebraically presented blue scheme $X$.
 \end{cor}

\subsection{Log schemes}
\label{subsection: log schemes}

Blueprints and blue schemes give rise to log schemes. More precisely, a blueprint defines a chart for an affine log structure. We will explain this connection in the following, with an emphasis on fine log schemes, which play a central role in logarithmic algebraic geometry.

\subsubsection{Fine log structures}

We recall the definition of fine log structures from Kato's paper \cite{Kato89}. As usual, monoids are commutative and written multiplicatively. A monoid morphism sends $1$ to $1$. A \emph{pre-log structure on a scheme $X$}\index{Pre-log structure} is a sheaf $\cM$ of monoids on the \'etale site $X_\et$ of $X$ together with a morphism $\alpha:\cM\to\cO_X$ of sheaves of monoids. A \emph{morphism $(X,\cM)\to(Y,\cN)$ of schemes $X$ and $Y$ with pre-log structures $\alpha:\cM\to\cO_X$ resp.\ $\beta:\cN\to\cO_Y$}\index{Morphism!of pre-log structures} is a pair $(f,h)$ of a morphism $f:X\to Y$ of schemes and a morphism $h:f^\ast(\cN)\to\cM$ of sheaves such that the diagram
\[
 \xymatrix{f^\ast(\cN) \ar[r]^h \ar[d]_{f^\ast\!\beta} & \cM\ar[d]^\alpha \\ f^\ast(\cO_Y) \ar[r] & \cO_X}
\]
commutes. A \emph{log structure}\index{Log structure} is a pre-log structure $\alpha:\cM\to\cO_X$ such that $\alpha^{-1}(\cO_X^\times)\to\cO_X^\times$ is an isomorphism.  Given a pre-log structure, we define the \emph{log structure associated with $\cM$} as the pushout $\cM^a$ of the diagram
\[
 \xymatrix{\alpha^{-1}(\cO_X^\times) \ar[r] \ar[d] & \cM \\ \cO_X^\times}.
\]
Then $\cM^a$ together with the induced morphism $\tilde\alpha:\cM^a\to\cO_X$ is indeed a log structure.

Let $P$ be a monoid. A monoid morphism $P\to \Gamma(X,\cO_X)$ is the same as a morphism $\cP_X\to\cO_X$ of sheaves of monoids where $\cP_X$ is the constant sheaf associated with $P$. A \emph{chart of a log structure $\alpha:\cM\to\cO_X$}\index{Chart of a log structure} is a monoid morphism $P\to\Gamma(X,\cO_X)$ such that $\alpha:\cM\to\cO_X$ is isomorphic to the log structure associated with the corresponding morphism $\cP_X\to\cO_X$ of sheaves of monoids.

A chart $P\to\Gamma(X,\cO_X)$ is \emph{finitely generated}\index{Chart of a log structure!finitely generated} if $P$ is so. A chart $P\to\Gamma(X,\cO_X)$ is \emph{integral}\index{Chart of a log structure!integral} if $P$ is so, i.e.\ $ab=ac$ implies $b=c$ for all elements $a,b,c\in P$. A chart $P\to\Gamma(X,\cO_X)$ is \emph{affine}\index{Chart of a log structure!affine} if $X$ is affine. 

A \emph{fine log structure}\index{Fine log structure} is a log structure that has \'etale locally affine finitely generated and integral charts. A \emph{(fine) log scheme}\index{Fine log scheme}\index{Log scheme} is a scheme together with a (fine) log structure. We denote the category of log schemes by $\Sch_\Z^\log$.

\subsubsection{General blueprints}

A \emph{general blueprint}\index{Blueprint!general} is a monoid $A$ together with an equivalence relation $\cR$ on $\N[A]$ that is additive and multiplicative. Note that this is Axiom \eqref{ax1} in the definition of a pre-addition in Section \ref{subsection: blueprints}. We do not impose Axioms \eqref{ax2} (existence of a zero) and \eqref{ax3} (properness). To avoid confusion, we adopt the convention of \cite{blueprints1} and call blueprints in the sense of this text \emph{proper blueprints with zero}\index{Blueprint!proper with zero} throughout the rest of Section \ref{subsection: log schemes}.

We extend the notation $B=\bpquot A\cR$ from proper blueprints with zero to general blueprints. A \emph{morphism $f:B_1\to B_2$ of general blueprints}\index{Morphism!of general blueprints} is a multiplicative map such that $\sum a_i\=\sum b_j$ in $B_1$ implies $\sum f(a_i)\=\sum f(b_j)$ in $B_2$. The universal ring $B^+_\Z$ of a general blueprint $B=\bpquot A\cR$ is defined in the same way as for proper blueprints with zero, namely, as the quotient of $\Z[A]$ by the ideal $\cI(\cR)=\{\sum a_i-\sum b_j|\sum a_i\=\sum b_j\}$ This comes together with a blueprint morphism $B\to B^+_\Z$ that sends an element $a$ of $B$ to its class in $B^+_\Z$.

As explained in \cite{blueprints1}, the theory of blue schemes extends to the definition of a \emph{general blue scheme}\index{Blue scheme!general}.

\subsubsection{The log scheme associated with a general blue scheme}

Let $B=\bpquot A\cR$ be a general blueprint and $X=\Spec B^+_\Z$. This defines the log structure $\alpha:\cM\to\cO_X$ associated with the monoid morphism $A\to B^+_\Z=\Gamma(X,\cO_X)$. 

If $X$ is a general blue scheme, then an open affine covering $\{U_i\}$ of $X$ by spectra of blueprints $B_i=\bpquot{A_i}{\cR_i}$ yield charts $A_i\to \Gamma(U_{i,\Z}^+,\cO_{U_{i,\Z}^+})$ for an associated log structure $\alpha_i:\cM_i\to \cO_{U_{i,\Z}^+}$ of $U_{i,\Z}^+$. These log structures glue to a log structure $\alpha:\cM\to X^+_\Z$ of $X$. This association is functorial, i.e.\ we obtain a functor $\log: \Sch_\Fun\to\Sch_\Z^\log$ from general blue schemes to log schemes. 

In order to characterize some properties of this functor, we introduce the following definitions. A general blueprint $B$ is \emph{strictly integral}\index{Blueprint!strictly integral} if $ab=ac$ implies $b=c$ for any $a,b,c\in B$. Note that a strictly integral blueprint cannot be proper with a zero, which explains the necessity to consider general blueprints in this context. A general blueprint $B$ is \emph{strictly finitely generated}\index{Blueprint!strictly finitely generated} if the underlying monoid is finitely generated. A blue scheme $X$ is \emph{strictly integral}\index{Blue scheme!strictly integral} (or \emph{locally strictly of finite type}\index{Blue scheme!strictly of finite type}) if for every affine open $U=\Spec B$ of $X$, the blueprint $B$ is strictly integral (resp.\ strictly finitely generated).

\begin{thm}
 The essential image of $\log: \Sch_\Fun\to\Sch_\Z^\log$ is the subcategory of log schemes with a Zariski local covering by charts. If $X$ is a general blue scheme that is strictly integral and locally strictly of finite type, then $\log (X)$ is a fine log scheme with a Zariski local covering by finitely generated and integral charts.
\end{thm}

\begin{rem} 
 Presumably, formally \'etale coverings can be defined for (general) blue schemes as well, which would allow us to extend $\log: \Sch_\Fun\to\Sch_\Z^\log$ to a functor whose essential image contains all fine log schemes.
  
 However, there are in general many different general blue schemes that give rise to the same log scheme. This means that passing from general blue schemes to log schemes loses information. Furthermore, the functor $\log:\Sch_\Fun\to\Sch_\Z^\log$ is not full. 
\end{rem}

\subsection{Congruence schemes}
\label{subsection: congruence schemes}

We conclude this text with some remarks on a possible definition of congruence schemes for blueprints. Note that Berkovich's definition of $\Fun$-schemes in \cite{Berkovich11} is based on congruences for monoids. Deitmar defines congruence schemes for \emph{sesquiads}, i.e.\ cancellative blueprints, in \cite{Deitmar11a}. Lescot defines the congruence spectrum of an idempotent semiring in \cite{Lescot11}. Note that these approaches have different technical realizations. In particular, Berkovich's and Deitmar's theory do not coincide in the case of monoids.

We extend the definition of the topological space of a congruence spectrum to blueprints. We review the different definitions of a structure sheaf (for the former two approaches) and add a third possible definition. It is not clear, which is the preferred approach to congruence schemes. This depends on a good behaviour of the global sections functor, see Remark \ref{rem: global sections for congruence schemes}.

\subsubsection{The congruence spectrum}
\label{subsubsection: the congruence spectrum}

Let $B=\bpquot A\cR$ be a blueprint. For an equivalence relation $\sim$ on $A$, we define the \emph{linear extension $\sim_\N$ of $\sim$ to $\N[A]$} as the equivalence relation on $\N[A]$ that is generated by $\sum a_i \sim_\N\sum b_i$ if $a_i\sim b_i$ for all $i$. We define the equivalence relation $\sim_\cR$ as the smallest equivalence relation containing both $\cR$ and $\sim_\N$.

A \emph{congruence on $B$}\index{Congruence} is an equivalence relation $\sim$ on $A$ that satisfies the following properties.
\begin{enumerate}
 \item The equivalence relation $\sim_\N$ is a pre-addition for $A$.
 \item The restriction of $\sim_\cR$ to $A$ equals $\sim$.
\end{enumerate}
Equivalently, an equivalence relation $\sim$ is a congruence if and only if it satisfies the following two conditions for all $a,b,c,d\in A$ (cf.\ \cite[Lemma 2.3]{blueprints1}).
\begin{enumerate}
 \item[{\eqref{ax1}$^*$}\!\!] \ If $a\sim b$ and $c\sim d$, then $ac\sim bd$.
 \item[{\eqref{ax2}$^*$}\!\!] \ If there exists a sequence
        $$ a \ \= \ \sum c_{1,k} \ \sim_\N \ \sum d_{1,k} \ \= \ \sum c_{2,k} \ \sim_\N \quad \dotsb \quad \sim_\N \ \sum d_{n,k} \ \= \ b $$
        with $c_{i,k},d_{i,k}\in A$, then $a\sim b$.
\end{enumerate}
Let $f:B\to C$ be a morphism of blueprints. The \emph{kernel of $f$}\index{Kernel} is the relation $\sim_f$ on $B$ that is defined by $a\sim_f b$ if and only if $f(a)\= f(b)$. Then we have the following correspondence between kernels and congruences, cf.\ Propositions 2.5 and 2.6 in \cite{blueprints1}.

\begin{prop}
 Let $f:B\to C$ be a morphism of blueprints. Then its kernel $\sim_f$ is a congruence. If $\sim$ is a congruence of $B$, then there exists a unique epimorphism $f:B\to C$ such that $\sim$ is the kernel of $f$.
\end{prop}

The morphism $f:B\to C$ is universal for $\sim$ in the sense that every morphism $g:B\to C'$ such that $\sim$ is contained in the kernel $\sim_g$ of $g$ factors uniquely through $f:B\to C$. This justifies to call $C$ the \emph{quotient of $B$ by $\sim$}\index{Blueprint!quotient} and denote $C$ by $B/\sim$. This means that the congruences of $B$ stay in bijection to the quotients of $B$.

A congruence is \emph{proper} if it has more than one equivalence class. A \emph{prime congruence of $B$}\index{Prime congruence} is a proper congruence $\sim$ such that for all $a,b,c\in B$ with $ab\sim ac$, either $b\sim c$ or $a\sim 0$. Equivalently, a congruence $\sim$ on $B$ is a prime congruence if and only if the quotient $B/\sim$ is \emph{integral}\index{Blueprint!integral}, i.e.\ $0\neq1$ and $ab=bc$ implies $b=c$ for all $a,b,c\in B$ with $a\neq0$.

The \emph{congruence spectrum $\CSpec B$ of $B$}\index{Congruence spectrum} is the set of all prime congruences on $B$ together with the topology generated by \emph{basis open subsets} 
\[
 U_{f,g} \quad = \quad \bigl\{ \ \sim\in\CSpec B \ \bigl| \ f\nsim g \ \bigr\}
\]
where $f$ and $g$ vary through $B$. If $f:B\to C$ is a morphism of blueprints, then the inverse image $f^{-1}(\sim)$ of a prime congruence $\sim$ on $C$ is a prime congruence on $B$. This defines a continuous map $f^\ast:\CSpec C\to \CSpec B$ of congruence spectra.

\subsubsection{Connection to the usual spectrum}
\label{subsubsubsection: connection to the usual spectrum}

The \emph{absorbing ideal of a congruence $\sim$ on $B$}\index{Absorbing ideal} is the subset $I_\sim=\{a\in B|a\sim 0\}$ of $B$. Let $I\subset B$ be an ideal and $\sim^I$ be the equivalence relation on $A$ that is defined by $a\sim^I b$ if and only if $a=b$ or $a,b\in I$. Let $\sim^I_\cR$ be the equivalence relation on $\N[A]$ as defined in the last section. Then the \emph{congruence generated by $I$} is the restriction $\sim_I$ of $\sim^I_\cR$ to $A$. The following proposition summarizes Propositions 2.13 and 2.14, Corollary 2.15 and Lemma 2.23 of \cite{blueprints1}.

\begin{prop}\label{prop: congruences and ideals}
 If $\sim$ is a congruence, then $I_\sim$ is equal to the ideal $f^{-1}(0)$ for the quotient map $f:B\to B/\sim$. If $I$ is an ideal, then $\sim_I$ is equal to the kernel of the quotient map $f:B\to B/I$. If $I$ is an ideal and $\sim=\sim_I$ the congruence generated by $I$, then $I_\sim=I$. If $\sim$ is a prime congruence, then $I_\sim$ is a prime ideal.
\end{prop}

If $B$ is a ring, then the associations $\sim\mapsto I_\sim$ and $I\mapsto\sim_I$ are mutually inverse and establish a correspondence between ideals and congruences. In general, it is only true that the vanishing ideal $I_\sim$ of the congruence $\sim$ generated an ideal $I$ is equal to $I$, but not vice versa. In other words, a blueprint has in general more congruences, and thus quotients, than ideals.

The above proposition implies the following connection between $\CSpec B$ and $\Spec B$.

\begin{lemma}
 Let $B$ be a blueprint. Then the association $\sim\mapsto I_\sim$ defines a surjective continuous map $\CSpec B\to\Spec B$. This map is functorial in $B$, i.e. a morphism of blueprints $f:B\to C$ gives rise to a commutative diagram
 \[
  \xymatrix{\CSpec C\ar[r]^{f^\ast} \ar[d] & \CSpec B \ar[d] \\ \Spec C\ar[r]^{f^\ast} & \Spec B }
 \]
 of topological spaces.
\end{lemma}

One can associate residue fields to the points of a congruence spectrum in the following way. The \emph{localization of $B$ at a prime congruence $\sim$}\index{Blueprint!localization} is the blueprint $B_\sim=B_\fp$ where $\fp=I_\sim$ is the absorbing ideal of $\sim$, which is prime by Proposition \ref{prop: congruences and ideals}. The \emph{residue field of $\sim$}\index{Residue field} is $B_\sim/\gen\sim_{B_\sim}$, which is a blue field.

\subsubsection{From congruence spectra to congruence schemes}
\label{subsubsection: from congruence spectra to congruence schemes}

The difficulty of defining congruence schemes is that there is not a clear way of how to define a structure sheaf for congruence spectra. In the following, we will describe three possible approaches, which are the generalization of Berkovich's definition of $\Fun$-scheme in \cite{Berkovich11} from monoids to blueprints, Deitmar's definition of congruence schemes in \cite{Deitmar11a} and a third independent idea.

The idea of Berkovich's approach is to restrict the collection of open sets to a class of open subsets for which a coordinate blueprint can be defined by a universal property. Namely, a basis open subset $U=U_{f,g}$ of $\CSpec B$ is an \emph{affine open (after Berkovich)} if the category of $B$-algebras $g:B\to C$ with $g^\ast(\CSpec C)\subset U$ has an initial object $B\to B_U$. The \emph{structure sheaf of $\CSpec B$ (after Berkovich)} is the sheaf supported by the topology generated by affine opens that is associated with the collection of coordinate algebras $B\to B_U$.

Deitmar's approach makes use of the equivalence of ideals and congruences for rings. However, to do so, we have to restrict to cancellative blueprints resp.\ sesquiads. Then it is possible to associate with every basis open $U_{f,g}$ the following blueprint $B_{f,g}$. The base extension $(U_{f,g})_\Z^+$ to schemes is the open affine subscheme of $X^+_\Z$ that is the complement of the vanishing set of $h=g-f$. Let $R_{f,g}=B_\Z^+[(h)^{-1}]$ be the coordinate ring of $(U_{f,g})_\Z^+$ and $S=\{h^i\}_{i\geq 0}$. The structure sheaf of $\CSpec B$ is supported on all open subsets of $\CSpec B$ and its value on $U_{f,g}$ is the subblueprint of the ring $R_{f,g}$ that is generated by the multiplicative subsets $B$ and $S^{-1}$ of $R_{f,g}$. Note that this definition doesn't cover the case of congruence spectra of semirings or more general non-cancellative blueprints, which are of interest for connections to tropical and analytic geometry.

A third possibility is to define an \emph{affine open subset of $\CSpec B$} as an open subset $U_{f,g}$ that is isomorphic to the congruence spectrum of a blueprint, i.e.\ there is a morphism $B\to B_{f,g}$ of blueprints that induces an homeomorphism of $\CSpec B_{f,g}$ with $U_{f,g}$ and isomorphisms between the residue fields of prime congruences. The \emph{structure sheaf of $\CSpec B$} is the sheaf supported on the topology generated by affine opens and that sends an affine open $U_{f,g}$ to $B_{f,g}$. Note that this approach is closely connected to Berkovich's idea, though it is not clear to me if it produces the same class of an affine opens and the same structure sheaf.

Either of the three approaches yields a locally blueprinted space $\CSpec B$ and a notion of affine congruence schemes. A \emph{congruence scheme}\index{Congruence scheme} can be defined as a locally blueprinted space that admits an open covering by subspaces that are isomorphic to affine congruence schemes.

\begin{rem}\label{rem: global sections for congruence schemes}
 A desirable property to pursue the theory beyond the basic definitions is that (a subcategory of) the category of affine congruence schemes is equivalent to (a subcategory of) the category of blueprints. Let $\Gamma$ be the endofunctor on blueprints that sends a blueprint $B$ to the global sections $\Gamma(X,\cO_X)$ of the congruence spectrum $X=\CSpec B$. Note that this functor depends on the definition of the structure sheaf. Deitmar shows in \cite[Thm.\ 2.5.1]{Deitmar11a} that $\Gamma B$ is isomorphic to $\Gamma\Gamma B$ for his notion of a structure sheaf. It is not clear to me whether a similar statement is true for the other two approaches.
 
\end{rem}

\newpage
\addcontentsline{toc}{section}{Index}
\printindex

\addcontentsline{toc}{section}{References}
\bibliographystyle{plain}

\end{document}